\documentclass[12pt]{amsart}
\usepackage{txfonts}      %{article} was 12pt latex e
\usepackage{amssymb}
\usepackage{eucal}
\usepackage{amsmath}
\usepackage{amscd}
\usepackage{xcolor}
\usepackage{multicol}
\usepackage[all]{xy}           %xypic macro for latex2.09
\usepackage{graphicx}
\usepackage{color}
\usepackage{colordvi}
\usepackage{xspace}
\usepackage{tikz}
\usepackage{makecell}
\usepackage{appendix}
\usepackage{enumerate,enumitem}
\usepackage{amsthm}
\usepackage[compress]{cite}

\usepackage{ifpdf}
\ifpdf
\usepackage[colorlinks,final,backref=page,hyperindex]{hyperref}
\else
\usepackage[colorlinks,final,backref=page,hyperindex,hypertex]{hyperref}
\fi

\usepackage[active]{srcltx} %SRC Specials for DVI Searching

\newtheorem{thm}{Theorem}[section]
\newtheorem{lem}[thm]{Lemma}

\newtheorem{pro}[thm]{Proposition}
\theoremstyle{definition}
\newtheorem{defi}[thm]{Definition}
\newtheorem{ex}[thm]{Example}

\newcommand {\emptycomment}[1]{} %to remove paragraphs

\newcommand{\nc}{\newcommand}
\newcommand{\delete}[1]{}

\nc{\tred}[1]{\textcolor{red}{#1}}
\nc{\tblue}[1]{\textcolor{blue}{#1}}
\nc{\tgreen}[1]{\textcolor{green}{#1}}
\nc{\tpurple}[1]{\textcolor{purple}{#1}}
\nc{\tgray}[1]{\textcolor{gray}{#1}}
\nc{\torg}[1]{\textcolor{orange}{#1}}
\nc{\tmag}[1]{\textcolor{magenta}}
\nc{\btred}[1]{\textcolor{red}{\bf #1}}
\nc{\btblue}[1]{\textcolor{blue}{\bf #1}}
\nc{\btgreen}[1]{\textcolor{green}{\bf #1}}
\nc{\btpurple}[1]{\textcolor{purple}{\bf #1}}

\nc{\tforall}{\ \ \text{for all }} \nc{\hatot}{\,\widehat{\otimes}
\,} \nc{\complete}{completed\xspace} \nc{\wdhat}[1]{\widehat{#1}}

\nc{\ts}{\mathfrak{p}} \nc{\mts}{c_{(i)}\ot d_{(j)}}

\nc{\NA}{{\bf NA}} \nc{\LA}{{\bf Lie}} \nc{\CLA}{{\bf CLA}}

\nc{\cybe}{CYBE\xspace} \nc{\nybe}{NYBE\xspace}
\nc{\ccybe}{CCYBE\xspace}

\nc{\ndend}{pre-Novikov\xspace} \nc{\calb}{\mathcal{B}}
\nc{\rk}{\mathrm{r}}

\nc{\vspa}{\vspace{-.1cm}} \nc{\vspb}{\vspace{-.2cm}}
\nc{\vspc}{\vspace{-.3cm}} \nc{\vspd}{\vspace{-.4cm}}
\nc{\vspe}{\vspace{-.5cm}}

\nc{\disp}[1]{\displaystyle{#1}}
\nc{\bin}[2]{ (_{\stackrel{\scs{#1}}{\scs{#2}}})}  %binomial coeff
\nc{\binc}[2]{ \left (\!\! \begin{array}{c} \scs{#1}\\
    \scs{#2} \end{array}\!\! \right )}  %binomial coeff
\nc{\bincc}[2]{  \left ( {\scs{#1} \atop
    \vspace{-.5cm}\scs{#2}} \right )}  %binomial coeff
\nc{\ot}{\otimes} \nc{\sot}{{\scriptstyle{\ot}}}
\nc{\otm}{\overline{\ot}}
\nc{\ola}[1]{\stackrel{#1}{\la}}%${\Bbb Z}$

\nc{\scs}[1]{\scriptstyle{#1}} \nc{\mrm}[1]{{\rm #1}}

\nc{\dirlim}{\displaystyle{\lim_{\longrightarrow}}\,}
\nc{\invlim}{\displaystyle{\lim_{\longleftarrow}}\,}

\nc{\bfk}{{\bf k}} \nc{\bfone}{{\bf 1}} \nc{\rpr}{\circ}
\nc{\dpr}{{\tiny\diamond}} \nc{\rprpm}{{\rpr}}

\nc{\calao}{{\mathcal A}} \nc{\cala}{{\mathcal A}}
\nc{\calc}{{\mathcal C}} \nc{\cald}{{\mathcal D}}
\nc{\cale}{{\mathcal E}} \nc{\calf}{{\mathcal F}}
\nc{\calfr}{{{\mathcal F}^{\,r}}} \nc{\calfo}{{\mathcal F}^0}
\nc{\calfro}{{\mathcal F}^{\,r,0}} \nc{\oF}{\overline{F}}
\nc{\calg}{{\mathcal G}} \nc{\calh}{{\mathcal H}}
\nc{\cali}{{\mathcal I}} \nc{\calj}{{\mathcal J}}
\nc{\call}{{\mathcal L}} \nc{\calm}{{\mathcal M}}
\nc{\caln}{{\mathcal N}} \nc{\calo}{{\mathcal O}}
\nc{\calp}{{\mathcal P}} \nc{\calq}{{\mathcal Q}}
\nc{\calr}{{\mathcal R}} \nc{\calt}{{\mathcal T}}
\nc{\caltr}{{\mathcal T}^{\,r}} \nc{\calu}{{\mathcal U}}
\nc{\calv}{{\mathcal V}} \nc{\calw}{{\mathcal W}}
\nc{\calx}{{\mathcal X}} \nc{\CA}{\mathcal{A}}

\nc{\fraka}{{\mathfrak a}} \nc{\frakB}{{\mathfrak B}}
\nc{\frakb}{{\mathfrak b}} \nc{\frakd}{{\mathfrak d}}
\nc{\oD}{\overline{D}} \nc{\frakF}{{\mathfrak F}}
\nc{\frakg}{{\mathfrak g}} \nc{\frakm}{{\mathfrak m}}
\nc{\frakM}{{\mathfrak M}} \nc{\frakMo}{{\mathfrak M}^0}
\nc{\frakp}{{\mathfrak p}} \nc{\frakS}{{\mathfrak S}}
\nc{\frakSo}{{\mathfrak S}^0} \nc{\fraks}{{\mathfrak s}}
\nc{\os}{\overline{\fraks}} \nc{\frakT}{{\mathfrak T}}
\nc{\oT}{\overline{T}}
\nc{\frakX}{{\mathfrak X}} \nc{\frakXo}{{\mathfrak X}^0}
\nc{\frakx}{{\mathbf x}}
\nc{\frakTx}{\frakT}      %All rooted trees, correspond to \ncsha(X)
\nc{\frakTa}{\frakT^a}        % rooted trees for \ncsha(A)
\nc{\frakTxo}{\frakTx^0}   % rooted trees for \ncshao(X)
\nc{\caltao}{\calt^{a,0}}   % rooted trees for \ncshao(A)
\nc{\ox}{\overline{\frakx}} \nc{\fraky}{{\mathfrak y}}
\nc{\frakz}{{\mathfrak z}} \nc{\oX}{\overline{X}}

\title[Non-abelian extensions and Wells exact sequences of Lie-Yamaguti algebras ]
{ Non-abelian extensions and Wells exact sequences of Lie-Yamaguti algebras }

\author{ Qinxiu Sun}
\address{Department of Mathematics, Zhejiang University of Science and Technology, Hangzhou, 310023} \email{qxsun@126.com}

\author{Zhen Li}
\address{Department of Mathematics, Zhejiang University of Science and Technology, Hangzhou, 310023}
         \email{lz200002272022@163.com}


\begin{document}
\begin{abstract}
The goal of the present paper is to investigate
 non-abelian extensions of Lie-Yamaguti algebras
 and explore extensibility of a pair of automorphisms about a non-abelian extension of Lie-Yamaguti algebras.
 First, we study non-abelian extensions of Lie-Yamaguti algebras
 and classify the non-abelian extensions in terms of non-abelian cohomology groups. Next, we characterize the non-abelian
 extensions in terms of Maurer-Cartan elements. Moreover, we discuss the equivalent conditions of the extensibility of
 a pair of automorphisms about a non-abelian extension of Lie-Yamaguti algebras,
and derive the fundamental sequences of Wells in the context of Lie-Yamaguti algebras. Finally,
we discuss the previous results in the case of abelian extensions of Lie-Yamaguti algebras.

\end{abstract}

\maketitle

\vspace{-1.2cm}

\tableofcontents

\vspace{-1.2cm}

\allowdisplaybreaks

\section{Introduction}
A Lie-Yamaguti algebra first appeared in Nomizu’s work on the affine invariant connections on homogeneous
spaces in 1950's \cite{10}, which was a generalization of Lie algebras and Lie triple systems.
In 1960’s, Yamaguti introduced an algebraic
structure and named it a general Lie triple system or a Lie triple algebra \cite{29,30}. Later,
the cohomology theory of this object was investigated in \cite{29} by Yamaguti.
In the study of Courant algebroids, Kinyon and Weinstein named general Lie triple systems by Lie–Yamaguti algebras \cite{15}.
 Since then, Lie-Yamaguti algebras have attracted much attention and are widely explored.
  For example, the irreducible modules of Lie–Yamaguti algebras were considered in
\cite{5, 28}, deformations and extensions of Lie-Yamaguti algebras were investigated in \cite{17,31}.
The relative Rota-Bxter operators on Lie-Yamaguti algebras and their cohomologies were considered in \cite{25, 8,9}.

Extensions are useful mathematical objects to understand the
underlying structures. The non-abelian extension is a relatively
general one among various extensions (e.g. central extensions,
abelian extensions, non-abelian extensions etc.). Non-abelian
extensions were first developed by Eilenberg and Maclane \cite{018},
which induced to the low dimensional non-abelian cohomology group.
Then numerous works have been devoted to non-abelian extensions of
various kinds of algebras, such as Lie (super)algebras, Leibniz algebras, Lie 2-algebras, Lie Yagamuti algebras,
associative conformal algebras, Rota-Baxter groups,
Rota-Baxter Lie algebras and Rota-Baxter Leibniz algebras, see \cite{ 08, 014, 048, 021, 023, 027,029, 028,047} and their references.
The abelian extensions of Lie Yagamuti algebras were considered
in \cite{021}. But little is known about the non-abelian extensions of Lie Yagamuti algebras. This is the first motivation for writing this paper.

Another interesting study related to extensions of algebraic structures is given by the extensibility or inducibility of a pair
of automorphisms. When a pair of automorphisms is inducible? This problem was first
 considered by Wells \cite{037} for abstract groups and further studied in \cite{045,032}.
Since then, several authors have studied this subject further, see \cite{047,021,023,026,028} and references therein.
The extensibility problem of a pair of derivations on abelian extensions was investigated in \cite{13,035}.
Recently, the extensibility problem of a pair of derivations and automorphisms was extended to
the context of abelian extensions of Lie coalgebras \cite{016}. As byproducts, the Wells short exact sequences were obtained for
various kinds of algebras \cite{014, 015,021, 023,026,045,028,047}, which connected the relative automorphism groups and
the non-abelian second cohomology groups. Inspired by these results, we investigate extensibility of a pair of automorphisms on a non-abelian
extension of Lie Yagamuti algebras. This is another motivation for writing the present paper. Moreover, we give necessary and sufficient conditions for
 a pair of automorphisms to be extensible, and derive the analogue of the Wells short exact sequences in the context of
  non-abelian extensions of Lie Yagamuti algebras.

The paper is organized as follows. In Section 2, we recall the definition of Lie-Yamaguti algebras
and their representations. We also recall some basic information about the cohomology groups of
Lie-Yamaguti algebras. In Section 3, we
investigate non-abelian extensions and classify the non-abelian
extensions using the non-abelian cohomology groups. In Section 4, we characterize equivalent non-abelian extensions
using Maurer-Cartan elements. In Section 5, we study
the extensibility problem of a pair of automorphisms about a non-abelian extension of Lie-Yamaguti algebras.
In Section 6, we derive Wells short exact sequences in the context of non-abelian
extensions of Lie-Yamaguti algebras. Finally,
we discuss the previous results in the case of abelian extensions of Lie-Yamaguti algebras.

Throughout the paper, let $k$ be a field. Unless otherwise
specified, all vector spaces and algebras are over $k$.

\setlength{\baselineskip}{1.25\baselineskip}

%%%%%%%%%%%%%%%%%%%%%%%%%%%%%%%%%%%%%%%%%%%%%%%%%%%%%%%%%%%%%%%%%%%%%%%%%%%%%%%%%%%%%%%%%%%%%%%%%%%%%%%

\section{Preliminaries on Lie-Yamaguti algebras}

We recall the notions of Lie-Yamaguti algebras, representations
and their cohomology theory. For the details see \cite{29,30}.

\begin{defi} A Lie-Yamaguti algebra is a vector space $\mathfrak g$ with a bilinear map $[ \cdot, \cdot]:
\mathfrak g \otimes \mathfrak g\longrightarrow \mathfrak g$ and a trilinear map
$\{ \ , \ , \ \}:\mathfrak g\otimes \mathfrak g\otimes \mathfrak g\longrightarrow \mathfrak g$,
satisfying
\begin{equation}\label{eq2.1}[x_1,x_2]+[x_2,x_1]=0,~~\{x_1,x_2,x_3\}+\{x_2,x_1,x_3\}=0,\end{equation}
\begin{equation}\label{eq2.3}[[x_1,x_2],x_3]+[[x_2,x_3],x_1]+[[x_3,x_1],x_2]+\{x_1,x_2,x_3\}+
\{x_2,x_3,x_1\}+\{x_3,x_1,x_2\}=0,\end{equation}
\begin{equation}\label{eq2.4}\{[x_1,x_2],x_3,y_1\}+\{[x_2,x_3],x_1,y_1\}+\{[x_3,x_1],x_2,y_1\}=0,\end{equation}
\begin{equation}\label{eq2.5}\{x_1,x_2,[y_1,y_2]\}=[\{x_1,x_2,y_1\},y_2]+[y_1,\{x_1,x_2,y_2\}],
\end{equation}
\begin{equation}\label{eq2.6}\{x_1,x_2,\{y_1,y_2,y_3\}\}=\{\{x_1,x_2,y_1\},y_2,y_3\}+\{y_1,\{x_1,x_2,y_2\},y_3\}+\{y_1,y_2,\{x_1,x_2,y_3\}\},\end{equation}
for all $x_1,x_2,x_3,y_1,y_2,y_3\in \mathfrak g$.
 Denote it by $(\mathfrak g,[  \  ,  \ ],\{  \  , \ , \ \})$ or simply by $\mathfrak g $.
\end{defi}

A subspace $L$ of $\mathfrak g$ is an ideal of $\mathfrak g$ if
$[L,\mathfrak g] \subseteq L, \{L, \mathfrak g,\mathfrak g \} \subseteq L$ and $\{\mathfrak g,\mathfrak g, L\} \subseteq L$.
 An ideal $L$ of $\mathfrak g$ is said to be an abelian ideal of $\mathfrak g$ if
$[L,L]=0$ and $\{L, L,\mathfrak g \} = \{\mathfrak g,L, L \} =\{L,\mathfrak g,L\}=0$.

\begin{defi} A representation of a Lie-Yamaguti algebra $\mathfrak g$ consists of a vector
space $V$ together with a linear map $\mu:\mathfrak g\longrightarrow {gl}(V)$
and bilinear maps $\theta,D:\mathfrak g\wedge \mathfrak g\longrightarrow {gl}(V)$
satisfying
\begin{equation}\label{eq2.9}\theta([x_1,x_2],x_3)=\theta(x_1,x_3)\mu(x_2)-\theta(x_2,x_3)\mu(x_1),\end{equation}
\begin{equation}\label{eq2.10}[D(x_1,x_2),\mu(y_1)]=\mu(\{x_1,x_2,y_1\}),\end{equation}
\begin{equation}\label{eq2.11}\theta(x_1,[y_1,y_2])=\mu(y_1)\theta(x_1,y_2)-\mu(y_2)\theta(x_1,y_1),\end{equation}
\begin{equation}\label{eq2.12}[D(x_1,x_2),\theta(y_1,y_2)]=\theta(\{x_1,x_2,y_1\},y_2)+\theta(y_1,\{x_1,x_2,y_2\}),\end{equation}
\begin{equation}\label{eq2.13}\theta(x_1,\{y_1,y_2,y_3\})
=\theta(y_2,y_3)\theta(x_1,y_1)-\theta(y_1,y_3)\theta(x_1,y_2)+D(y_1,y_2)\theta(x_1,y_3),\end{equation}
\begin{equation}\label{eq2.7}D(x_1,x_2)-\theta(x_2,x_1)+\theta(x_1,x_2)+\mu([x_1,x_2])-[\mu(x_1),\mu(x_2)]=0,\end{equation}
for all $x_1,x_2,x_3,y_1,y_2,y_3\in \mathfrak g$. Denote the representation of $\mathfrak g$ by
$(V,\mu,\theta,D)$ or simply by $V$.\end{defi}

When $(V,\mu,\theta,D)$ is a representation
 of $\mathfrak g$, by a direct computation, the following conditions are
 also satisfied:
\begin{equation}\label{eq2.8}D([x_1,x_2],x_3)+D([x_2,x_3],x_1)+D([x_3,x_1],x_2)
=0,\end{equation}
\begin{equation}\label{eq2.70}D(\{x_1,x_2,x_3\},x_4)+D(x_3,\{x_1,x_2,x_4\})=[D(x_1,x_2),D(x_3,x_4)]
,\end{equation}
\begin{equation}\label{eq2.71}\theta(\{y_1,y_2,y_3\},x_1)=\theta(y_1,x_1)\theta(y_3,y_2)-\theta(y_2,x_1)\theta(y_3,y_1)
-\theta(y_3,x_1)D(y_1,y_2).\end{equation}

\begin{pro} Let $(\mathfrak g,[  \  ,  \ ]_{\mathfrak g},\{  \  , \ , \ \}_{\mathfrak g})$ be
a Lie-Yamaguti algebra and $V$ a vector space. Assume that $\mu:\mathfrak g\longrightarrow {gl}(V)$
is a linear map and $\theta,D:\mathfrak g\wedge \mathfrak g\longrightarrow {gl}(V)$ are bilinear maps.
Then $(V,\mu,\theta,D)$ is a representation of $\mathfrak g$ if and only
if $(\mathfrak g\oplus V,[  \  ,  \ ],\{  \  , \ , \ \})$ is a Lie-Yamaguti algebra, where
\begin{equation*}\{x+u,y+v,z+w\}=\{x,y,z\}_{\mathfrak g}+\theta(y,z)u-\theta(x,z)v+D(x,y)w,\end{equation*}
and
\begin{equation*}[x+u,y+v]=[x,y]_{\mathfrak g}+\mu(x)v-\mu(y)u,\end{equation*}
for all $x,y,z\in T,u,v,w\in V$.
The Lie-Yamaguti algebra $(\mathfrak g\oplus V,[  \  ,  \ ],\{  \  , \ , \ \})$ is called
the semidirect product Lie-Yamaguti algebra. Denote it simply by $\mathfrak g\ltimes V$.
\end{pro}

\begin{ex} Let $\mathfrak g$ be a Lie-Yamaguti algebra. Define ${ad}:\mathfrak g\longrightarrow {gl}(\mathfrak g),
~R,L:\mathfrak g \wedge \mathfrak g\longrightarrow {gl}(\mathfrak g)$
respectively by ${ad}(x)(y)=[x,y],R(x,y)(z)=\{z,x,y\}$ and $L(x,y)(z)=\{x,y,z\}$. Then
$(\mathfrak g,{ad},R,L)$ is a representation of $\mathfrak g$, which is called the adjoint
representation.\end{ex}

Next, we recall the cohomology of Lie-Yamaguti algebras following \cite{30}.
Let $\mathfrak g$ be a Lie-Yamaguti algebra and $(V,\mu,\theta,D)$
be its representation. The cochain complex is given as follows:
\begin{itemize}[leftmargin=15pt]
\item Set $C^{1}(\mathfrak g,V)=\mathrm{Hom} (\mathfrak g,V)$ and $C^{0}(\mathfrak g,V)$ be the subspace spanned by the diagonal elements $(f,f)\in C^{1}(\mathfrak g,V)\times C^{1}(\mathfrak g,V)$.
\item (For $n\geq 2$)~Set $C^{n}(\mathfrak g,V)=\mathrm{Hom} (\otimes^{n}\mathfrak g,V)$, where the space of all
$n$-linear maps $f\in C^{n}(\mathfrak g,V)$ satisfying
\begin{equation*}f(x_1,\cdot\cdot\cdot,x_{2i-1},x_{2i},\cdot\cdot\cdot,x_{n})=0,~~\hbox{if} ~x_{2i-1}=x_{2i},~\forall~i=1,2,\cdot\cdot\cdot,[\frac{n}{2}].\end{equation*}
\end{itemize}

Then, for all $n\geq 1$, put the $(2n,2n+1)$-cochain groups
\begin{equation*}C^{(2n,2n+1)}(\mathfrak g,V)=C^{2n}(\mathfrak g,V)\times
C^{2n+1}(\mathfrak g,V)\end{equation*}
and
\begin{equation*}C^{(3,4)}(\mathfrak g,V)=C^{3}(\mathfrak g,V)\times
C^{4}(\mathfrak g,V).\end{equation*}
Define the coboundary operator $\delta=(\delta_I,\delta_{II})$ in the following cochain
complex of $\mathfrak g$ with coefficents in $V$:
\[\xymatrix@C=20pt@R=20pt{C^{0}(\mathfrak g,V)\ar[r]^-{\delta}&C^{(2,3)}(\mathfrak g,V)\ar[d]_-{\delta^{*}}\ar[r]^-\delta&C^{(4,5)}(\mathfrak g,V)\ar[r]^-\delta&C^{(6,7)}(\mathfrak g,V)
\ar[r]^-\delta&\cdot\cdot\cdot
\\&C^{(3,4)}(\mathfrak g,V)&&&.}\]
For $n\geq 1$, the coboundary operator $\delta=(\delta_I,\delta_{II}):C^{(2n,2n+1)}(\mathfrak g,V)\longrightarrow C^{(2n+2,2n+3)}(\mathfrak g,V)$ is
 defined by
\begin{eqnarray*}&&(\delta_I f)(x_1,\cdot\cdot\cdot,x_{2n+2})\\&=&
\mu(x_{2n+1})g(x_1,\cdot\cdot\cdot,x_{2n},x_{2n+2})-\mu(x_{2n+2})g(x_1,\cdot\cdot\cdot,x_{2n+1})-
g(x_1,\cdot\cdot\cdot,x_{2n},[x_{2n+1},x_{2n+2}])\\&&+
\sum_{k=1}^{n}(-1)^{n+k+1}D(x_{2k-1},x_{2k})f(x_1,\cdot\cdot\cdot,x_{2k-2},x_{2k+1},\cdot\cdot\cdot,x_{2n+2})
\\&&+
\sum_{k=1}^{n}\sum_{j=2k+1}^{2n+2}(-1)^{n+k}f(x_1,\cdot\cdot\cdot,x_{2k-2},x_{2k+1},\cdot\cdot\cdot,\{x_{2k-1},x_{2k},x_{j}\},\cdot\cdot\cdot,x_{2n+2})
\end{eqnarray*}
and
\begin{eqnarray*}&&(\delta_{II} g)(x_1,\cdot\cdot\cdot,x_{2n+3})\\&=&
\theta(x_{2n+2},x_{2n+3})g(x_1,\cdot\cdot\cdot,x_{2n+1})-\theta(x_{2n+1},x_{2n+3})g(x_1,\cdot\cdot\cdot,x_{2n},x_{2n+2})
\\&&+\sum_{k=1}^{n+1}(-1)^{n+k+1}D(x_{2k-1},x_{2k})g(x_1,\cdot\cdot\cdot,x_{2k+1},\cdot\cdot\cdot,x_{2n+3})
\\&&+
\sum_{k=1}^{n+1}\sum_{j=2k+1}^{2n+3}(-1)^{n+k}g(x_1,\cdot\cdot\cdot,x_{2k+1},\cdot\cdot\cdot,\{x_{2k-1},x_{2k},x_{j}\},\cdot\cdot\cdot,x_{2n+3})
\end{eqnarray*}
for any pair $(f,g)\in C^{(2n,2n+1)}(\mathfrak g,V)$ and
$x_1,\cdot\cdot\cdot,x_{2n+3}\in \mathfrak g$. And when $n=0$, the coboundary
operator
$$\delta=(\delta_I,\delta_{II}):
C^{0}(\mathfrak g,V)\longrightarrow C^{2}(\mathfrak g,V)\times C^{3}(\mathfrak g,V)$$ is defined
 as follows:
$$(\delta_I f)(x_1,x_{2})=
\mu(x_{1})f(x_2)-\mu(x_{2})f(x_1)-
f([x_{1},x_{2}]),
$$
$$(\delta_{II} f)(x_1,x_2,x_{3})=
\theta(x_2,x_{3})f(x_1)-\theta(x_{1},x_{3})f(x_{2})
+D(x_{1},x_{2})f(x_{3})-
f(\{x_{1},x_{2},x_{3}\})
$$ for any $f\in C^{1}(\mathfrak g,V)$ and $x_1,x_2,x_{3}\in \mathfrak g$.

Define $\delta^{*}=(\delta_{I}^{*},\delta_{II}^{*}): C^{(2,3)}(\mathfrak g,V)\longrightarrow  C^{(3,4)}(\mathfrak g,V)$ by
\begin{align*}\delta_{I}^{*}f(x,y,z)=&-\mu(x)f(y,z)-\mu(y)f(z,x)-\mu(z)f(x,y)+f([x,y],z)+f([y,z],x)
\\&+f([z,x],y)+g(x,y,z)+g(y,z,x)+g(z,x,y),
\end{align*}
\begin{align*}\delta_{II}^{*}f(x,y,z,w)=&\theta(x,w)f(y,z)+\theta(y,w)f(z,x)+\theta(z,w)f(x,y)+g([x,y],z,w)
\\&+g([y,z],x,w)+g([z,x],y,w),
\end{align*}
for all $(f,g)\in C^{(2,3)}(\mathfrak g,V)$ and $x,y,z,w\in \mathfrak g$.

Denote the set of all the $(2n,2n+1)$-cocycles and the
$(2n,2n+1)$-coboundaries, respectively by $Z^{(2n,2n+1)}(\mathfrak g, V )$ and  $B^{(2n,2n+1)}(\mathfrak g, V ) $.
In particular,
\begin{equation*}Z^{(2,3)}(\mathfrak g, V )=\{(f,g)\in C^{(2,3)}(\mathfrak g, V )|\delta(f,g)=0,\delta^{*}(f,g)=0\},\end{equation*}
\begin{equation*}B^{(2,3)}(\mathfrak g, V )=\{(\delta_{I}(f),\delta_{II}(f)) |f\in C^{1}(\mathfrak g, V )\}.\end{equation*}
Define
$H^{1}(\mathfrak g, V )= \{f\in C^{1}(\mathfrak g, V )|\delta_{I}(f)=0,\delta_{II}(f)=0\} $, which is called the first
cohomology group of $\mathfrak g$ with coefficients in
the representation $V$ and define
$$H^{(2n,2n+1)}(\mathfrak g, V )= Z^{(2n,2n+1)}(\mathfrak g, V )/B^{(2n,2n+1)}(\mathfrak g, V ),~~n\geq 1$$ which is
called the (2n,2n+1)-cohomology group of $\mathfrak g$ with coefficients in
the representation $V$.

%%%%%%%%%%%%%%%%%%%%%%%%%%%%%%%%%%%%%%%%%%%%%%%%%%%%%%%%%%%%%%%%%%%%%%%%%%%%%%%%%%%%%%%%%%%%%%%%%%%%%%%

\section{Non-abelian
extensions and non-abelian (2,3)-cocycles of Lie-Yamaguti algebras}

In this section, we are devoted to considering non-abelian
extensions and non-abelian (2,3)-cocycles of Lie-Yamaguti algebras.

\begin{defi} Let $\mathfrak g$ and $\mathfrak h$ be two Lie-Yamaguti algebras. A non-abelian
extension of $\mathfrak g$ by $\mathfrak h$ is a Lie-Yamaguti algebra $\hat{\mathfrak g}$,
which fits into a short exact sequence of Lie-Yamaguti algebras
$$\mathcal{E}:0\longrightarrow\mathfrak h\stackrel{i}{\longrightarrow} \hat{\mathfrak g}\stackrel{p}{\longrightarrow}\mathfrak g\longrightarrow0.$$
When $\mathfrak h$ is an abelian ideal of $\hat{\mathfrak g}$, the extension $\mathcal{E}$ is called an
abelian extension of $\mathfrak g$ by $\mathfrak h$.
Denote an extension as above simply by $\hat{\mathfrak g}$ or $\mathcal{E}$. A section of $p$ is a linear map $s: \mathfrak g\longrightarrow  \hat{\mathfrak g} $
such that $ps = I_{\mathfrak g}$.
\end{defi}

\begin{defi}
Let $ \hat{\mathfrak g}_1$ and $\hat{\mathfrak g}_2$
be two non-abelian extensions of $\mathfrak g$ by $\mathfrak h$. They are said to be
equivalent if there is a homomorphism of Lie-Yamaguti algebras
$f:\hat{\mathfrak g}_1\longrightarrow \hat{\mathfrak g}_2$ such that
the following commutative diagram holds:
 \begin{equation}\label{Ene1} \xymatrix{
  0 \ar[r] & \mathfrak h\ar@{=}[d] \ar[r]^{i_1} & \hat{\mathfrak g}_1\ar[d]_{f} \ar[r]^{p_1} & \mathfrak g \ar@{=}[d] \ar[r] & 0\\
 0 \ar[r] & \mathfrak h \ar[r]^{i_2} & \hat{\mathfrak g}_2 \ar[r]^{p_2} & \mathfrak g  \ar[r] & 0
 .}\end{equation}
\end{defi}
Denote by $\mathcal{E}_{nab}(\mathfrak g,\mathfrak h)$ the set of all non-abelian extensions of $\mathfrak g$ by $\mathfrak h$.

Next, we define a non-abelian cohomology group and show that
the non-abelian extensions are classified by the non-abelian cohomology groups.

\begin{defi} Let $\mathfrak g$ and $\mathfrak h$ be two Lie-Yamaguti algebras.
A non-abelian (2,3)-cocycle on $\mathfrak g$ with values in
 $\mathfrak h$ is a septuple $(\chi,\omega,\mu,\theta,D,\rho,T)$
of maps such that $\omega:\mathfrak g\otimes\mathfrak g\otimes\mathfrak
g\longrightarrow \mathfrak h$ is trilinear, $\chi:\mathfrak g\otimes\mathfrak g\longrightarrow \mathfrak h,
~\theta,D:\mathfrak g\wedge \mathfrak
g\longrightarrow \mathfrak{gl}(\mathfrak h)$ are bilinear
and $\mu:\mathfrak g\longrightarrow \mathfrak{gl}(\mathfrak h) ,
~\rho,T:\mathfrak g\longrightarrow
\mathrm{Hom} (\mathfrak h\wedge \mathfrak h,\mathfrak h)$ are linear,
and the following five parts of identities are satisfied for all
$x_i,y_i, x, y, z\in \mathfrak g~ (i=1,2,3), a,b,c\in \mathfrak h$,
\begin{itemize}[leftmargin=15pt]
	\item Those resembling Eq.~(\ref{eq2.1}):
\begin{equation}\label{L00}\chi([x,y]_{\mathfrak
g}+\chi(y,x)
=0,~~\omega(x,y,z)+\omega(y,x,z)=0,
\end{equation}
\begin{equation}\label{L01}D(x,y)a+D(y,x)a=0,~T(x)(a,b)+T(x)(b,a)=0
.\end{equation}
\item Those resembling Eq.~(\ref{eq2.3}):
\begin{align}\label{L12}&\chi([x,y]_{\mathfrak
g},z)-\mu(z)\chi(x,y)+\omega(x,y,z)
+\chi([y,z]_{\mathfrak
g},x)-\mu(x)\chi(y,z)+\omega(y,z,x)
\nonumber\\&+\chi([z,x]_{\mathfrak
g},y)-\mu(y)\chi(z,x)+\omega(z,x,y)
=0,\end{align}
\begin{equation}\label{L13}\mu([x,y]_{\mathfrak
g})a+[\chi(x,y),a]_{\mathfrak
h}+D(x,y)a-\mu(x)\mu(y)a-\theta(y,x)a+\mu(y)\mu(x)a+\theta(x,y)a
=0,\end{equation}
\begin{equation}\label{L14}[\mu(x)a,b]_{\mathfrak
h}+\rho(x)(a,b)
-\mu(x)[a,b]_{\mathfrak
h}+T(x)(a,b)-[\mu(x)b,a]_{\mathfrak
h}-\rho(x)(b,a)=0.
 \end{equation}
\item Those resembling Eq.~(\ref{eq2.4}):
\begin{equation}\label{L15}\theta(z,w)\chi(x,y)+\theta(x,w)\chi(y,z)+\theta(y,w)\chi(z,x)=0,
 \end{equation}
\begin{align}\label{L16}&D([x,y]_{\mathfrak
g},z)a-\rho(z)(\chi(x,y),a)
+D([y,z]_{\mathfrak
g},x)a-\rho(x)(\chi(y,z),a)
\nonumber\\&+D([z,x]_{\mathfrak
g},y)a-\rho(y)(\chi(z,x),a)=0,
 \end{align}
\begin{equation}\label{L17}\theta([x,y]_{\mathfrak
g},z)a-T(z)(\chi(x,y),a)
-\theta(x,z)\mu(y)
a+\theta(y,z)\mu(x)
a=0,\end{equation}
\begin{equation}\label{L18}\rho([x,y]_{\mathfrak
g})(a,b)+\{\chi(x,y),a,b\}_{\mathfrak
h}-\rho(x)(\mu(y)a,b)+\rho(y)(\mu(x)a,b)=0,
 \end{equation}
\begin{equation}\label{L19}T(y)(\mu(x)a,b)+\theta(x,y)[a,b]_{\mathfrak
h}-T(y)(\mu(x)b,a)=0,
 \end{equation}
\begin{equation}\label{L20}\{\mu(x)a,b,c\}_{\mathfrak
h}-\rho(x)([a,b]_{\mathfrak
h},c)-\{\mu(x)b,a,c\}_{\mathfrak
h}=0,\end{equation}
\begin{equation}\label{L21}T(x)([a,b]_{\mathfrak
h},c)+T(x)([b,c]_{\mathfrak
h},a)+T(x)([c,a]_{\mathfrak
h},b)=0.\end{equation}
\item Those resembling Eq.~(\ref{eq2.5}):
\begin{equation}\label{L22}D(x,y)\chi(z,w)=\chi(\{x,y,z\}_{\mathfrak
g},w)-\mu(w)\omega(x,y,z)+\mu(z)\omega(x,y,w)+\chi(z,\{x,y,w\}_{\mathfrak
g}),\end{equation}
\begin{equation}\label{L23}D(x,y)\mu(z)a=\mu(\{x,y,z\}_{\mathfrak
g})a+[\omega(x,y,z),a]_{\mathfrak
h}+\mu(z)D(x,y)a,
\end{equation}
\begin{equation}\label{L24}\theta(x,[y,z]_{\mathfrak
g})a-\rho(x)(a,\chi(y,z))=\mu(y)\theta(x,z)a-\mu(z)\theta(x,y)a,
\end{equation}
\begin{equation}\label{L25}D(x,y)[a,b]_{\mathfrak
h}=[D(x,y)a,b]_{\mathfrak
h}+[a,D(x,y)b]_{\mathfrak
h},\end{equation}
\begin{equation}\label{L26}\rho(x)(a,\mu(y)b)+[\theta(x,y)a,b]_{\mathfrak
h}-\mu(y)\rho(x)(a,b)=0,
\end{equation}
\begin{equation}\label{L27}T([x,y]_{\mathfrak
g})(a,b)+\{a,b,\chi(x,y)\}_{\mathfrak
h}-\mu(x)T(y)(a,b)+\mu(y)T(x)(a,b)=0,
\end{equation}
\begin{equation}\label{L28}
\{a,b,\mu(x)c\}_{\mathfrak
h}=\mu(x)\{a,b,c\}_{\mathfrak
h}-[c,T(x)(a,b)]_{\mathfrak
h},\end{equation}
\begin{equation}\label{L29}
\rho(x)(a,[b,c]_{\mathfrak h})=[\rho(x)(a,b),c]_{\mathfrak h}+[b,\rho(x)(a,c)]_{\mathfrak h}.
\end{equation}
\item Those resembling Eq.~(\ref{eq2.6}):\\
\begin{align}\label{L1}&D(x_1,x_2)\omega(y_1,y_2,y_3)+\omega(x_1,x_2,\{y_1,y_2,y_3\}_{\mathfrak
g})=\omega(\{x_1,x_2,y_1\}_{\mathfrak
g},y_2,y_3)\nonumber\\&+\theta(y_2,y_3)\omega(x_1,x_2,y_1)+\omega(y_1,\{x_1,x_2,y_2\}_{\mathfrak
g},y_3)-\theta(y_1,y_3)\omega(x_1,x_2,y_2)\nonumber\\&+\omega(y_1,y_2,\{x_1,x_2,y_3\}_{\mathfrak
g})+D(y_1,y_2)\omega(x_1,x_2,y_3),
 \end{align}
\begin{align}\label{L2}&D(x,y)\theta(z,w)a-\theta(z,w)D(x,y)a\nonumber\\=&\theta(\{x,y,z\}_{\mathfrak
g},w)a+\theta(z,\{x,y,w\}_{\mathfrak g})a-T(w)(\omega(x,y,z),a)-\rho(z)(a,\omega(x,y,w)),
 \end{align}
 \begin{equation}\label{L3}\theta(x,\{y,z,w\}_{\mathfrak
g})a-\rho(x)(a,\omega(y,z,w))=\theta(z,w)\theta(x,y)a-\theta(y,w)\theta(x,z)a+D(y,z)\theta(x,w)a,
 \end{equation}
 \begin{align}\label{L30}&D(x,y)D(z,w)a-D(z,w)D(x,y)a\nonumber\\=&
D(\{x,y,z\}_{\mathfrak g},w)a+D(z,\{x,y,w \}_{\mathfrak g})a-\rho(w)(\omega(x,y,z),a)+\rho(z)(\omega(x,y,w),a),
\end{align}
\begin{align}\label{L4}(D(x,y)\rho(z)-\rho(\{x,y,z\}_{\mathfrak
g}))(a,b)&=\rho(z)((D(x,y)a,b)+(a,D(x,y)b))+\{\omega(x,y,z),a,b\}_{\mathfrak
h},
 \end{align}
\begin{equation}\label{L5}D(y,z)(\rho(x)(a,b))=\rho(x)(a,D(y,z)b)-\rho(z
)(\theta(x,y)a,b)+\rho(y)(\theta(x,z)a,b),
 \end{equation}
\begin{equation}\label{L6}\theta(y,z)(\rho(x)(a,b))=\rho(x)(a,\theta(y,z)b)-T(z
)(\theta(x,y)a,b)-\rho(y)(b,\theta(x,z)a),
 \end{equation}
 \begin{align}\label{L31}
(D(x,y)T(z)-T(\{x,y,z\}_{\mathfrak
g}))(a,b)=T(z)((D(x,y)a,b)
+(a,D(x,y)b))+\{a,b,\omega(x,y,z)\}_{\mathfrak h},
 \end{align}
\begin{align}\label{L32}
(T(\{x,y,z\}_{\mathfrak
g})- D(x,y)T(z))(a,b)+\{a,b,\omega(x,y,z)\}_{\mathfrak h}=(\theta(y,z)T(x)-\theta(x,z)T(y))(a,b),
 \end{align}
\begin{equation}\label{L7}D(x,y)(\{a,b,c\}_{\mathfrak
h})=\{D(x,y)a,b,c\}_{\mathfrak
h}+\{a,D(x,y)b,c\}_{\mathfrak
h}+\{a,b,D(x,y)c\}_{\mathfrak h},
 \end{equation}
\begin{equation}\label{L8}\rho(x)(a,\rho(y)(b,c))=\rho(y)(\rho(x)(a,b),c)+\rho(y)(b,\rho(x)(a,c))
-\{\theta(x,y)a,b,c\}_{\mathfrak h},
 \end{equation}
\begin{equation}\label{L9}\{a,b,\theta(x,y)c\}_{\mathfrak h}=\theta(x,y)\{a,b,c\}_{\mathfrak
h}-T(y)(T(x)(a,b),c)-\rho(x)(c,T(y)(a,b)),
 \end{equation}
 \begin{equation}\label{L33}
\rho(x)(a,T(y)(b,c))=T(y)(\rho(x)(a,b),c)+T(y)(b,\rho(x)(a,c))
-\{a,c,\theta(x,y)a\}_{\mathfrak
h},\end{equation}
\begin{equation}\label{L34}
\{a,b,D(x,y)c\}_{\mathfrak h}=D(x,y)\{a,b,c\}_{\mathfrak h}-\rho(y)(T(x)(a,b),c),
+\rho(x)(T(y)(a,b),c), \end{equation}
\begin{equation}\label{L10}\{a,b,\rho(x)(c,d)\}_{\mathfrak h}=\rho(x)(\{a,b,c\}_{\mathfrak
h},d)-\{c,T(x)(a,b),d\}_{\mathfrak
h}+\rho(x)(c,\{a,b,d\}_{\mathfrak h}),
 \end{equation}
\begin{equation}\label{L11}\rho(x)(a,\{b,c,d\}_{\mathfrak
h})=\{\rho(x)(a,b),c,d\}_{\mathfrak h}+\{b,\rho(x)(a,c),d\}_{\mathfrak
h}+\{b,c,\rho(x)(a,d)\}_{\mathfrak h}.
 \end{equation}
\begin{equation}\label{L35}
\{a,b,T(x)(c,d)\}_{\mathfrak h}
=T(x)(\{a,b,c\}_{\mathfrak h},d)
+T(x)(c,\{a,b,d\}_{\mathfrak h})
+\{c,d,T(x)(a,b)\}_{\mathfrak h}. \end{equation}
\end{itemize}
\end{defi}

\begin{defi} Let $(\chi_1,\omega_1,\mu_1,\theta_1,D_{1},\rho_1,T_1)$ and
$(\chi_2,\omega_2,\mu_2,\theta_2,D_{2},\rho_2,T_2)$ be two
non-abelian (2,3)-cocycles on $\mathfrak g$ with values in
 $\mathfrak h$. They are said to be equivalent if there
exists a linear map $\varphi:\mathfrak g\longrightarrow\mathfrak h$
such that for all $x, y, z\in \mathfrak g$ and $a,b\in \mathfrak h$,
the following equalities hold:
 \begin{equation}\label{E1}\chi_1(x,y)-\chi_2(x,y)=[\varphi(x),\varphi(y)]_{\mathfrak h}+\varphi[x,y]_{\mathfrak g}
 -\mu_{2}(x)\varphi(y)+\mu_{2}(y)\varphi(x),
 \end{equation}
\begin{align}\label{E2}&\omega_1(x,y,z)-\omega_2(x,y,z)=\theta_{2}(x,z)\varphi(y)-D_{2}(x,y)\varphi(z)
+\rho_2(x)(\varphi(y),\varphi(z))-\theta_{2}(y,z)\varphi(x)
\nonumber\\&+T_{2}(z)(\varphi(x),\varphi(y))-\rho_2(y)(\varphi(x),\varphi(z))-\{\varphi(x),\varphi(y),\varphi(z)\}_{\mathfrak
h}+\varphi\{x,y,z\}_{\mathfrak
g},
 \end{align}
\begin{equation}\label{E3}\mu_1(x)a-\mu_2(x)a=[a,\varphi(x)]_{\mathfrak
h},\end{equation}
   \begin{equation}\label{E4}\theta_1(x,y)a-\theta_2(x,y)a=\rho_{2}(x)(a,\varphi(y))-T_{2}(y)(a,\varphi(x))+\{a,\varphi(x),\varphi(y)\}_{\mathfrak
h},
 \end{equation}
 \begin{equation}\label{E5}D_{1}(x,y)a-D_{2}(x,y)a=\rho_2(y)(\varphi(x),a)-\rho_{2}(x)(\varphi(y),a)+\{\varphi(x),\varphi(y),a\}_{\mathfrak
h},
 \end{equation}
\begin{equation}\label{E6}\rho_1(x)(a,b)-\rho_2(x)(a,b)=\{a,\varphi(x),b\}_{\mathfrak
h}, ~~T_{1}(x)(a,b)-T_{2}(x)(a,b)=\{b,a,\varphi(x)\}_{\mathfrak
h}.\end{equation}
 \end{defi}
For convenience, we abbreviate a non-abelian (2,3)-cocycle $(\chi,\omega,\mu,\theta,D,\rho,T)$ as $(\chi,\omega)$, denote
the equivalent class of a non-abelian (2,3)-cocycle $(\chi,\omega,\mu,\theta,D,\rho,T)$ simply by $[(\chi,\omega)]$.
Furthermore, we denote the set of equivalent classes of non-abelian (2,3)-cocycles by
$H_{nab}^{(2,3)}(\mathfrak g,\mathfrak h)$.

Using the above notations, we define multilinear maps $[  \ ,  \ ]_{\chi}$ and $[  \ , \ , \ ]_{\omega}$ on $\mathfrak g\oplus \mathfrak h$ by
\begin{align}\label{NLts0}[x+a,y+b]_{\chi}&=[x,y]_{\mathfrak g}+\chi(x,y)+\mu(x)b-\mu(y)a+[a,b]_{\mathfrak
h}
,\end{align}
\begin{align}\label{NLts}\{x+a,y+b,z+c\}_{\omega}=&\{x,y,z\}_{\mathfrak g}+\omega(x,y,z)+D(x,y)c+\theta(y,z)a-\theta(x,z)b
\nonumber\\&+T(z)(a,b)+\rho(x)(b,c)-\rho(y)(a,c)+\{a,b,c\}_{\mathfrak
h}\end{align}
for all $x,y,z\in \mathfrak g$ and $a,b,c\in \mathfrak h$.

\begin{pro} \label{LY} With the above notions,
$(\mathfrak g\oplus \mathfrak h,[  \  ,  \ ]_{\chi},\{  \  , \ , \
\}_{\omega})$ is a Lie-Yamaguti algebra if and only if the septuple $(\chi,\omega,\mu,\theta,D,\rho,T)$
is a non-abelian (2,3)-cocycle. Denote this Lie-Yamaguti algebra
 $(\mathfrak g\oplus \mathfrak h,[  \  ,  \ ]_{\chi},\{  \  , \ , \
\}_{\omega})$ simply by $\mathfrak g\oplus_{(\chi,\omega)}\mathfrak h$.
\end{pro}

\begin{proof}
$(\mathfrak g\oplus \mathfrak h,[  \  ,  \
]_{\chi},\{  \  , \ , \
\}_{\omega})$ is a Lie-Yamaguti algebra if and only if
Eqs.~(\ref{eq2.1})-(\ref{eq2.6}) hold for $[  \  ,  \
]_{\chi},\{ \ , \ , \
\}_{\omega}$. In fact, it is easy to check that
Eq.~(\ref{eq2.1}) holds for $[  \  ,  \
]_{\chi},\{ \ , \ , \ \}_{\omega}$ if
and only if (\ref{L00})-(\ref{L01}) hold. In the following, we always assume that
$x,y,z,w\in \mathfrak g$ and $a,b,c,d\in \mathfrak h$.

 For the Eq.~(\ref{eq2.3}), we discuss it for the following cases:
 for all $x_1,x_2,x_3\in \mathfrak g\oplus \mathfrak h$,
\begin{enumerate}[label=$(\Roman*)$]
\item when all of the three elements $x_1,x_2,x_3$ belong to
$\mathfrak g$, (\ref{eq2.3}) holds if and only if (\ref{L12})
holds.

\item when $x_1,x_2,x_3$ equal to:
\begin{enumerate}[label=$(\roman*)$]
\item $x,y,a $ or $a,x,y $ or $x,a,y $ respectively, (\ref{eq2.3}) holds if and only if (\ref{L13}) holds.
\item $a,b,x$ or $a,x,b $ or $x,a,b $ respectively, (\ref{eq2.3}) holds
for if and only if (\ref{L14}) holds.
\end{enumerate}
\end{enumerate}
For the Eq.~(\ref{eq2.4}), we discuss it for the following cases:
 for all $x_1,x_2,x_3,y_1\in \mathfrak g\oplus \mathfrak h$,
\begin{enumerate}[label=$(\Roman*)$]
\item when all of the four elements $x_1,x_2,x_3,y_1$ belong to
$\mathfrak g$, (\ref{eq2.4}) holds if and only if (\ref{L15})
holds.

\item when $x_1,x_2,x_3,y_1$  equal to:
\begin{enumerate}[label=$(\roman*)$]
\item $x,y,z,a $, (\ref{eq2.4}) holds if and only if (\ref{L16}) holds.

\item $x,y,a,z$ or $x,a,y,z $ or $a,x,y,z $ respectively, (\ref{eq2.4}) holds
 if and only if (\ref{L17}) holds.

\item $x,y,a,b$ or $x,a,y,b $ or $a,x,y,b $ respectively, (\ref{eq2.4}) holds
 if and only if (\ref{L18}) holds.

\item $x,a,b,y$ or $a,x,b,y $ or $a,b,x,y $ respectively, (\ref{eq2.4}) holds
 if and only if (\ref{L19}) holds.

\item $x,a,b,c $ or $a,x,b,c $ or $a,b,x,y,c $
respectively, (\ref{eq2.4}) holds if and only if (\ref{L20}) holds.

\item $a,b,c,x $, (\ref{eq2.4}) holds if and only if (\ref{L21}) holds.
\end{enumerate}
\end{enumerate}
For the Eq.~(\ref{eq2.5}), we discuss it for the following cases:
 for all $x_1,x_2,y_1,y_2\in \mathfrak g\oplus \mathfrak h$,
\begin{enumerate}[label=$(\Roman*)$]
\item when all of the four elements $x_1,x_2,y_1,y_2$ belong to
$\mathfrak g$, (\ref{eq2.5}) holds if and only if (\ref{L22})
holds.

\item when $x_1,x_2,y_1,y_2$ equal to:
\begin{enumerate}[label=$(\roman*)$]
\item $x,y,z,a $ or $x,y,a,z $, (\ref{eq2.5}) holds if and only if (\ref{L23}) holds.

\item $x,a,y,z$ or $a,x,y,z $ respectively, (\ref{eq2.5}) holds if and only if (\ref{L24}) holds.

\item $x,y,a,b$, (\ref{eq2.5}) holds if and only if (\ref{L25}) holds.

\item $x,a,y,b$ or $a,x,y,b $ or $x,a,b,y $ or $a,x,b,y $ respectively, (\ref{eq2.5}) holds
 if and only if (\ref{L26}) holds.

\item $a,b,x,y $, (\ref{eq2.5}) holds if and only if (\ref{L27}) holds.

\item $a,b,c,x $ or $a,b,x,c $ respectively, (\ref{eq2.5}) holds if and only if (\ref{L28}) holds.

\item $a,x,b,c $ or $x,a,b,c $ respectively, (\ref{eq2.5}) holds if and only if (\ref{L29}) holds.
\end{enumerate}\end{enumerate}

For the Eq.~(\ref{eq2.6}), we discuss it for the following cases:
 for all $x_1,x_2,y_1,y_2,y_3\in \mathfrak g\oplus \mathfrak h$,
\begin{enumerate}[label=$(\Roman*)$]
\item when all of the five elements $x_1,x_2,y_1,y_2,y_3$ belong to
$\mathfrak g$, (\ref{eq2.6}) holds if and only if (\ref{L1})
holds.

\item when $x_1,x_2,y_1,y_2,y_3$ equal to:
\begin{enumerate}[label=$(\roman*)$]
\item $x,y,z,w,a $, (\ref{eq2.6}) holds if and only if (\ref{L30}) holds.

\item $x,y,z,a,w$ or $x,y,a,z,w $ respectively, (\ref{eq2.6}) holds if and only if (\ref{L2}) holds.

\item $x,a,y,z,w$ or $a,x,y,z,w $ respectively, (\ref{eq2.6}) holds if and only if (\ref{L3}) holds.

\item $x,y,z,a,b$ or $x,y,a,y,b $, (\ref{eq2.6}) holds
 if and only if (\ref{L4}) holds.

\item $x,a,y,z,b $ or $a,x,y,z,b $ respectively, (\ref{eq2.6}) holds if and only if (\ref{L5}) holds.

\item $x,y,a,b,z $, (\ref{eq2.6}) holds if and only if (\ref{L31}) holds.

\item $x,a,y,b,z $ or $a,x,y,b,z $ or $x,a,b,y,z $ or $a,x,b,y,z $ respectively, (\ref{eq2.5}) holds if and only if (\ref{L6}) holds.

\item $a,b,x,y,z $, (\ref{eq2.6}) holds if and only if (\ref{L32}) holds.

\item $x,y,a,b,c $, (\ref{eq2.6}) holds if and only if (\ref{L7}) holds.

\item $x,a,y,b,c $ or $a,x,y,b,c $ or $a,x,b,y,c $ or $x,a,b,y,c $ respectively, (\ref{eq2.6}) holds if and only if (\ref{L8}) holds.

\item $x,a,b,c,y $ or $a,x,b,c,y $ respectively, (\ref{eq2.6}) holds if and only if (\ref{L33}) holds.

\item $a,b,x,c,y $ or $a,b,c,x,y $ respectively, (\ref{eq2.6}) holds if and only if (\ref{L9}) holds.

\item $a,b,x,y,c $, (\ref{eq2.6}) holds if and only if (\ref{L34}) holds.

\item $a,b,c,d,x $, (\ref{eq2.6}) holds if and only if (\ref{L35}) holds.

\item $a,b,c,x,d $ or $a,b,x,c,d $ respectively, (\ref{eq2.6}) holds if and only if (\ref{L10}) holds.

\item $a,x,b,c,d $ or $x,a,b,c,d $ respectively, (\ref{eq2.6}) holds if and only if (\ref{L11}) holds.
\end{enumerate}\end{enumerate}
This completes the proof.
\end{proof}
Let
 $\mathcal{E}:0\longrightarrow\mathfrak h\stackrel{i}{\longrightarrow} \hat{\mathfrak g}\stackrel{p}{\longrightarrow}\mathfrak g\longrightarrow0$
be a non-abelian extension of $\mathfrak g$ by
$\mathfrak h$ with a section $s$ of $p$.
Define $\chi_{s}:\mathfrak g\otimes \mathfrak g\longrightarrow \mathfrak h,~\omega_{s}:\mathfrak g\otimes \mathfrak g\otimes\mathfrak
g\longrightarrow \mathfrak h,~\mu_{s}:\mathfrak g
\longrightarrow \mathfrak{gl}(\mathfrak h),~\theta_{s},D_{s}:\mathfrak g\wedge \mathfrak
g\longrightarrow \mathfrak{gl}(\mathfrak h),~\rho_{s},T_{s}:\mathfrak g\longrightarrow
\mathrm{Hom} (\mathfrak h\wedge \mathfrak h,\mathfrak h)$
 respectively by
 \begin{equation}\label{C0}\chi_{s}(x,y)=[s(x),s(y)]_{\hat{\mathfrak g}}-s[x,y]_{\mathfrak g},\end{equation}
\begin{equation}\label{C1}\omega_{s}(x,y,z)=\{s(x),s(y),s(z)\}_{\hat{\mathfrak g}}-s\{x,y,z\}_{\mathfrak g},\end{equation}
\begin{equation}\label{C2}\theta_{s}(x,y)a=\{a,s(x),s(y)\}_{\hat{\mathfrak g}},~~~~\rho_{s}(x)(a,b)=\{s(x),a,b\}_{\hat{\mathfrak g}},\end{equation}
\begin{equation}\label{C3}D_{s}(x,y)a=\{s(x),s(y),a\}_{\hat{\mathfrak g}},~~~~
T_{s}(x)(a,b)=\{a,b,s(x)\}_{\hat{\mathfrak
g}}\end{equation}
for any $x,y,z\in \mathfrak g,a,b\in \mathfrak h$.

By direct computations, we have
\begin{pro} \label{CY} With the above notions, $(\chi_{s},\omega_{s},\mu_{s},\theta_{s},D_{s},\rho_{s},T_{s})$ is a
non-abelian (2,3)-cocycle on $\mathfrak g$ with values in
 $\mathfrak h$. We call it the non-abelian (2,3)-cocycle corresponding to the extension $\mathcal{E}$ induced by $s$.
 Naturally, $(\mathfrak g\oplus\mathfrak h, [ \ , \ ]_{\chi_{s}},\{ \ , \ , \ \}_{\omega_{s}})$ is
 a Lie-Yamaguti algebra. Denote this Lie-Yamaguti algebra simply by $\mathfrak g\oplus_{(\chi_{s},\omega_{s})}\mathfrak h$.
\end{pro}
In the following, we denote $(\chi_{s},\omega_{s},\mu_{s},\theta_{s},D_{s},\rho_{s},T_{s})$
by $(\chi,\omega,\mu,\theta,D,\rho,T)$ without ambiguity.

 \begin{lem} \label{Le1} Let $(\chi_i,\omega_i,\mu_i,\theta_i,D_i,\rho_i,T_i)$ be the non-abelian (2,3)-cocycle
 corresponding to the extension $\mathcal{E}$ induced by $s_i$~(i=1,2).
 Then $(\chi_1,\omega_1,\mu_1,\theta_1,D_1,\rho_1,T_1)$ and $(\chi_2,\omega_2,\mu_2,\theta_2,D_2,\rho_2,T_2)$
 are equivalent, that is, the equivalent classes of non-abelian (2,3)-cocycles corresponding to
 a non-abelian extension
induced by a section are independent on the choice of sections.
 \end{lem}

\begin{proof}
Let $\hat{\mathfrak g}$ be a non-abelian extension
of $\mathfrak g$ by $\mathfrak h$. Assume that $s_1$ and $s_2$ are two
different sections of $p$, $(\chi_1,\omega_1,\mu_1,\theta_1,D_{1},\rho_1,T_1)$ and
$(\chi_2,\omega_2,\mu_2,\theta_2,D_{2},\rho_2,T_{2})$ are the corresponding non-abelian
(2,3)-cocycles. Define a linear map $\varphi:\mathfrak
g\longrightarrow\mathfrak h$ by $\varphi(x)=s_2(x)-s_1(x)$. Since
$p\varphi(x)=ps_2(x)-ps_1(x)=0$, $\varphi$ is well defined. Thanks to
Eqs.~(\ref{C1})-(\ref{C3}), we get
\begin{eqnarray*}&&w_1(x,y,z)=\{s_1(x),s_1(y),s_1(z)\}_{\hat{\mathfrak g}}-s_{1}\{x,y,z\}_{\mathfrak g}\\&=&
\{s_2(x)-\varphi(x),s_2(y)-\varphi(y),s_2(z)-\varphi(z)\}_{\hat{\mathfrak
g}}-(s_{2}\{x,y,z\}_{\mathfrak g}-\varphi(\{x,y,z\}_{\mathfrak g})
\\&=&\{s_2(x),s_2(y),s_2(z)\}_{\hat{\mathfrak g}}-\{s_2(x),\varphi(y),s_2(z)\}_{\hat{\mathfrak
g}}-\{s_2(x),s_2(y),\varphi(z)\}_{\hat{\mathfrak
g}}+\{s_2(x),\varphi(y),\varphi(z)\}_{\hat{\mathfrak
g}}\\&&-\{\varphi(x),s_2(y),s_2(z)\}_{\hat{\mathfrak
g}}+\{\varphi(x),\varphi(y),s_2(z)\}_{\hat{\mathfrak
g}}+\{\varphi(x),s_2(y),\varphi(z)\}_{\hat{\mathfrak
g}}-\{\varphi(x),\varphi(y),\varphi(z)\}_{\hat{\mathfrak
g}}\\&&-s_2\{x,y,z\}_{\mathfrak g}+\varphi\{x,y,z\}_{\mathfrak g}
\\&=&w_2(x,y,z)+\theta_2(x,z)\varphi(y)-D_{2}(x,y)\varphi(z)
+\rho_2(x)(\varphi(y),\varphi(z))-\theta_2(y,z)\varphi(x)+T_{2}(z)(\varphi(x),\varphi(y))
\\&&-\rho_2(y)(\varphi(x),\varphi(z)) -\{\varphi(x),\varphi(y),\varphi(z)\}_{\hat{\mathfrak
g}}+\varphi\{x,y,z\}_{\mathfrak
g},\end{eqnarray*} which yields that
 Eq.~(\ref{E2}) holds. Similarly, Eqs.~(\ref{E1}) and (\ref{E3})-(\ref{E6})
 hold. This finishes the proof.
\end{proof}

According to Proposition \ref{LY} and Proposition \ref{CY}, given a non-abelian extension
 $\mathcal{E}:0\longrightarrow\mathfrak h\stackrel{i}{\longrightarrow} \hat{\mathfrak g}\stackrel{p}{\longrightarrow}\mathfrak g\longrightarrow0$
of $\mathfrak g$ by
$\mathfrak h$ with a section $s$ of $p$, we have a non-abelian (2,3)-cocycle
 $(\chi_{s},\omega_{s},\mu_{s},\theta_{s},D_{s},\rho_{s},T_{s})$ and a Lie-Yamaguti algebra $\mathfrak g\oplus_{(\chi_s,\omega_s)} \mathfrak h$.
 It follows that
$\mathcal{E}_{(\chi_s,\omega_s)}:0\longrightarrow\mathfrak h\stackrel{i}{\longrightarrow} \mathfrak g\oplus_{(\chi_s,\omega_s)} \mathfrak h\stackrel{\pi}{\longrightarrow}\mathfrak g\longrightarrow0$ is a non-abelian extension of $\mathfrak g$ by $\mathfrak h$. Since any element
$\hat{w}\in \hat{\mathfrak g}$ can be written as $\hat{w}=a+s(x)$ with $a\in \mathfrak h,x\in \mathfrak g$,
define a linear map
\begin{equation*} f:\hat{\mathfrak g}\longrightarrow \mathfrak g\oplus_{(\chi_s,\omega_s)} \mathfrak h,~f(\hat{w})=f(a+s(x))=a+x.\end{equation*}
It is easy to check that $f$ is an isomorphism of Lie-Yamaguti algebras such that
the following commutative diagram holds:
 \begin{equation*} \xymatrix{
  \mathcal{E}:0 \ar[r] & \mathfrak h\ar@{=}[d] \ar[r]^-{i} & \hat{\mathfrak g}\ar[d]_-{f} \ar[r]^-{p} & \mathfrak g \ar@{=}[d] \ar[r] & 0\\
 \mathcal{E}_{(\chi_s,\omega_s)}:0 \ar[r] & \mathfrak h \ar[r]^-{i} & \mathfrak g\oplus_{(\chi_s,\omega_s)} \mathfrak h \ar[r]^-{\pi} & \mathfrak g  \ar[r] & 0,}\end{equation*}
 which indicates that the non-abelian extensions $\mathcal{E}$ and $\mathcal{E}_{(\chi_s,\omega_s)}$ of $\mathfrak g$ by
$\mathfrak h$ are equivalent. On the other hand, if $(\chi,\omega,\mu,\theta,D,\rho,T)$
 is a non-abelian (2,3)-cocycle on $\mathfrak g$ with values in
 $\mathfrak h$, there is a Lie-Yamaguti algebra $\mathfrak g\oplus_{(\chi,\omega)} \mathfrak h$, which yields the following
 non-abelian extension of $\mathfrak g$ by $\mathfrak h$:
 \begin{equation*}\mathcal{E}_{(\chi,\omega)}:0\longrightarrow\mathfrak h\stackrel{i}{\longrightarrow}\mathfrak g\oplus_{(\chi,\omega)} \mathfrak h
\stackrel{\pi}{\longrightarrow}\mathfrak g\longrightarrow0,\end{equation*}
where $i$ is the inclusion and $\pi$ is the projection.

In the following, we focus on the relationship between non-abelian (2,3)-cocycles and extensions.

\begin{pro}
 Let $\mathfrak g$ and $\mathfrak h$ be two Lie-Yamaguti algebras.
 Then the equivalent classes of non-abelian extensions of $\mathfrak g$ by $\mathfrak h$
are classified by the non-abelian cohomology group, that is,
 $\mathcal{E}_{nab}(\mathfrak g,\mathfrak h)\simeq H_{nab}^{(2,3)}(\mathfrak g,\mathfrak h)$.
 \end{pro}

 \begin{proof}
 Define a linear map
 \begin{equation*}\Theta:\mathcal{E}_{nab}(\mathfrak g,\mathfrak h)\rightarrow H_{nab}^{(2,3)}(\mathfrak g,\mathfrak h),~\end{equation*}
where $\Theta$ assigns an equivalent class of non-abelian extensions to the class of corresponding non-abelian (2,3)-cocycles.
First, we check that $\Theta$ is well-defined.
 Assume that $\mathcal{E}_1$ and $\mathcal{E}_2$ are two equivalent non-abelian extensions of $\mathfrak g$ by $\mathfrak h$ via the map
 $f$, that is, the commutative diagram (\ref{Ene1}) holds. Let $s_1:\mathfrak g\rightarrow \hat{\mathfrak g}_1$ be a
section of $p_1$. Then $p_2fs_1=p_1s_1=I_{\mathfrak g}$, which
follows that $s_2=fs_1$ is a section of $p_2$. Let
 $(\chi_1,\omega_1,\mu_1,\theta_1,D_{1},\rho_1,T_{1})$ and
$(\chi_2,\omega_2,\mu_2,\theta_2,D_{2},\rho_2,T_{2})$ be two non-abelian (2,3)-cocycles induced
by the sections $s_1,s_2$ respectively. Then we have,
\begin{eqnarray*}\theta_1(x,y)a&=&f(\theta_1(x,y)a)=f(\{a,s_1(x),s_1(y)\}_{\hat{\mathfrak g}_1})
\\&=&\{f(a),fs_1(x),fs_1(y)\}_{\hat{\mathfrak g}_2}\\&=&\{a,s_2(x),s_2(y)\}_{\hat{\mathfrak
g}_2}\\&=&\theta_2(x,y)a .\end{eqnarray*} By the same token, we have
\begin{equation*}D_{1}(x,y)a=D_{2}(x,y)a,\chi_1(x,y)=\chi_2(x,y),\omega_1(x,y,z)=\omega_2(x,y,z),\end{equation*}
\begin{equation*}\rho_1(x)(a,b)=\rho_2(x)(a,b),T_{1}(x)(a,b)=T_{2}(x)(a,b).\end{equation*} Thus,
$(\chi_1,\omega_1,\mu_1,\theta_1,D_{1},\rho_1,T_{1})=(\chi_2,\omega_2,\mu_2,\theta_2,D_{2},\rho_2,T_{2})$,
which means that $\Theta$ is well-defined.

 Next, we verify that $\Theta$ is injective. Indeed, suppose that
 $\Theta([\mathcal{E}_1])=[(\chi_1,\omega_1)]$ and $\Theta([\mathcal{E}_2])=[(\chi_2,\omega_2)]$. If
the equivalent classes $[(\chi_1,\omega_1)]=[(\chi_2,\omega_2)]$, we obtain that the non-abelian (2,3)-cocycles
 $(\chi_1,\omega_1,\mu_1,\theta_1,D_{1},\rho_1,T_{1})$ and
$(\chi_2,\omega_2,\mu_2,\theta_2,D_{2},\rho_2,T_{2})$ are equivalent
via the linear map $\varphi:\mathfrak g\longrightarrow
\mathfrak h$, satisfying Eqs.~~(\ref{E1})-(\ref{E6}). Define a linear map
$f:\mathfrak g\oplus_{(\chi_1,\omega_1)} \mathfrak h\longrightarrow \mathfrak g\oplus_{(\chi_2,\omega_2)}
\mathfrak h$ by
\begin{equation*}f(x+a)=x-\varphi(x)+a,~~\forall~x\in \mathfrak g,a\in \mathfrak h.\end{equation*}
 According to Eq.~~(\ref{NLts}), for all $x,y,z\in \mathfrak
g,a,b,c\in \mathfrak h$, we get
\begin{eqnarray*}&&f(\{x+a,y+b,z+c\}_{\omega_1})
\\&=&f(\{x,y,z\}_{\mathfrak g}+\omega_1(x,y,z)+D_{1}(x,y)c+\theta_1(y,z)a-\theta_1(x,z)b
\\&&+T_{1}(z)(a,b)+\rho_1(x)(b,c)-\rho_1(y)(a,c)+\{a,b,c\}_{\mathfrak h})
\\&=&\{x,y,z\}_{\mathfrak g}-\varphi(\{x,y,z\}_{\mathfrak g})+\omega_1(x,y,z)+D_{1}(x,y)c+\theta_1(y,z)a-\theta_1(x,z)b
\\&&+T_{1}(z)(a,b)+\rho_1(x)(b,c)-\rho_1(y)(a,c)+\{a,b,c\}_{\mathfrak h},\end{eqnarray*}
and
\begin{eqnarray*}&&\{f(x+a),f(y+b),f(z+c)\}_{\omega_2}
\\&=&\{x-\varphi(x)+a,y-\varphi(y)+b,z-\varphi(z)+c\}_{\omega_2}
\\&=&\{x,y,z\}_{\mathfrak g}+\omega_2(x,y,z)+D_{2}(x,y)(c-\varphi(z))+\theta_2(y,z)(a-\varphi(x))-\theta_2(x,z)(b-\varphi(y))
\\&&+T_{2}(z)(a-\varphi(x),b-\varphi(y))+\rho_2(x)(b-\varphi(y),c-\varphi(z))-\rho_2(y)(a-\varphi(x),c-\varphi(z))
\\&&+\{a-\varphi(x),b-\varphi(y),c-\varphi(z)\}_{\mathfrak h})
\\&=&\{x,y,z\}_{\mathfrak g}+\omega_2(x,y,z)+D_{2}(x,y)(c-\varphi(z))+\theta_2(y,z)(a-\varphi(x))-\theta_2(x,z)(b-\varphi(y))
\\&&+T_{2}(z)(\varphi(x),\varphi(y))-T_{2}(z)(a,\varphi(y))-T_{2}(z)(\varphi(x),b)+T_{2}(z)(a,b)
\\&&+\rho_2(x)(\varphi(y),\varphi(z))-\rho_2(x)(\varphi(y),c)-\rho_2(x)(b,\varphi(z))+\rho_2(x)(b,c)
-\rho_2(y)(\varphi(x),\varphi(z))\\&&+\rho_2(y)(\varphi(x),c)+\rho_2(y)(a,\varphi(z))-\rho_2(y)(a,c)
-\{\varphi(x),\varphi(y),\varphi(z)\}_{\mathfrak h}+\{\varphi(x),\varphi(y),c\}_{\mathfrak h}\\&&+\{\varphi(x),b,\varphi(z)\}_{\mathfrak h}
-\{\varphi(x),b,c\}_{\mathfrak h}+\{a,\varphi(y),\varphi(z)\}_{\mathfrak h}-\{a,b,\varphi(z)\}_{\mathfrak h}-\{a,\varphi(y),c\}_{\mathfrak h}
+\{a,b,c\}_{\mathfrak h}
.\end{eqnarray*}
In view of Eqs.~(\ref{E1})-(\ref{E6}), we have
$f(\{x+a,y+b,z+c\}_{\omega_1})=\{f(x+a),f(y+b),f(z+c)\}_{\omega_2}$. Similarly, $f([x+a,y+b]_{\chi_1})=[f(x+a),f(y+b)]_{\chi_2}$.
Hence, $f$ is a homomorphism of Lie-Yamaguti algebras. Clearly, the
following commutative diagram holds:
\begin{equation}
\xymatrix{
 \mathcal{E}_{(\chi_1,\omega_1)}: 0 \ar[r] & \mathfrak h\ar@{=}[d] \ar[r]^-{i} & \mathfrak g\oplus_{(\chi_1,\omega_1)} \mathfrak h \ar[d]_-{f} \ar[r]^-{\pi} & \mathfrak g \ar@{=}[d] \ar[r] & 0\\
\mathcal{E}_{(\chi_2,\omega_2)}: 0 \ar[r] & \mathfrak h \ar[r]^-{i} & \mathfrak g\oplus_{(\chi_2,\omega_2)} \mathfrak h \ar[r]^-{\pi} & \mathfrak g  \ar[r] & 0
 .}\end{equation}
Thus $\mathcal{E}_{(\chi_1,\omega_1)}$ and $\mathcal{E}_{(\chi_2,\omega_2)}$ are equivalent
non-abelian extensions of $\mathfrak g$ by $\mathfrak h$,
which means that $[\mathcal{E}_{(\chi_1,\omega_1)}]=[\mathcal{E}_{(\chi_2,\omega_2)}]$. Thus, $\Theta$ is injective.

 Finally, we claim that $\Theta$ is surjective.
 For any equivalent class of non-abelian (2,3)-cocycles $[(\chi,\omega)]$, by Proposition \ref{LY}, there is
 a non-abelian extension of $\mathfrak g$ by $\mathfrak h$:
   \begin{equation*}\mathcal{E}_{(\chi,\omega)}:0\longrightarrow\mathfrak h\stackrel{i}{\longrightarrow} \mathfrak g\oplus_{(\chi,\omega)} \mathfrak h\stackrel{\pi}{\longrightarrow}\mathfrak g\longrightarrow0.\end{equation*}
   Therefore, $\Theta([\mathcal{E}_{(\chi,\omega)}])=[(\chi,\omega)]$, which follows that $\Theta$ is surjective.
   In all, $\Theta$ is bijective. This finishes the proof.

 \end{proof}

%%%%%%%%%%%%%%%%%%%%%%%%%%%%%%%%%%%%%%%%%%%%%%%%%%%%%%%%%%%%%%%%%%%%%%%%%%%%%%%%%%%%%%%%%%%%%%%%%%%%%%%
\section{Non-abelian extensions in terms of Maurer-Cartan elements}
 In this section, we classify the non-abelian extensions using Maurer-Cartan
 elements. We start with recalling the Maurer-Cartan elements from \cite {049}.

Let $(L=\oplus_{i}L_{i},[ \ , \ ], d)$ be a differential graded Lie
algebra. The set $\mathrm{MC}(L)$ of Maurer-Cartan elements of $(L,[ \ , \ ],
d)$ is defined by
$$\mathrm{MC}(L)=\{\eta\in L_1|d \eta+\frac{1}{2}[\eta,\eta]=0\}.$$
Moreover, $\eta_0,\eta_1\in \mathrm{MC}(L)$ are called gauge equivalent if
and only if there exists an element $\varphi\in L_{0}$ such that
\begin{equation*}\eta_1=e^{ad_{\varphi}}\eta_{0}-\frac{e^{ad_{\varphi}}-1}{ad_{\varphi}}d\varphi.\end{equation*}

Let $\mathfrak g$ be a vector space. Denote by
\begin{equation*}\mathbb{C}(\mathfrak g,\mathfrak g)
=\mathrm{Hom}(
\wedge^{2}\mathfrak g\otimes\mathfrak g,\mathfrak g)\times \mathrm{Hom}(\wedge^{2}\mathfrak g\otimes\wedge^{2}\mathfrak g,\mathfrak g)\end{equation*}
and
\begin{equation}\label{DLY}
	\mathcal{C}^{n}(\mathfrak g,\mathfrak g)=\left\{
	\begin{aligned}
		&\mathrm{Hom}(\mathfrak g,\mathfrak g),&n=0,\\
		&\mathrm{Hom}(\underbrace{
\wedge^{2}\mathfrak g\otimes \cdot\cdot\cdot \otimes\wedge^{2}\mathfrak g }_{n},\mathfrak g)\times \mathrm{Hom}(\underbrace{\wedge^{2}\mathfrak g\otimes \cdot\cdot\cdot \otimes\wedge^{2}\mathfrak g}_n \otimes\mathfrak
g,\mathfrak g),&n\geq 1,
	\end{aligned}
	\right.
\end{equation}
Then $\mathcal{L}^{*}(\mathfrak g,\mathfrak g)=\mathcal{C}^{*}(\mathfrak g,\mathfrak g)\oplus \mathbb{C}(\mathfrak g,\mathfrak g)=\oplus_{n\geq
0}\mathcal{C}^{n}(\mathfrak g,\mathfrak g)\oplus \mathbb{C}(\mathfrak g,\mathfrak g)$, where the degree of elements
in $\mathcal{C}^{n}(\mathfrak g,\mathfrak g)$ is $n$, the degree of elements
in $\mathbb{C}(\mathfrak g,\mathfrak g)$ is $1$ and
$f\in \mathrm{Hom}(\otimes^{2n+1}\mathfrak
g,\mathfrak g)$ satisfying
\begin{align}\label{Co1}f(x_1,\cdot\cdot\cdot,x_{2i-1},x_{2i},\cdot\cdot\cdot,x_{n})=0,~~\hbox{if} ~x_{2i-1}=x_{2i},~\forall~i=1,2,\cdot\cdot\cdot,[\frac{n}{2}].\end{align}
For all $P=(P_I,P_{II})\in \mathcal{C}^{p}(\mathfrak g,\mathfrak g),
Q=(Q_I,Q_{II})\in \mathcal{C}^{q}(\mathfrak g,\mathfrak g)~(p,q\geq 1)$, denote by
\begin{equation*}P\circ Q=((P\circ Q)_{I},(P\circ Q)_{II})\in \mathcal{C}^{p+q}(\mathfrak g,\mathfrak g).\end{equation*}
The definition of $P\circ Q$ is given in \cite {8}. In detail,
\begin{align*}&(P\circ Q)_{I}(X_1,\cdot\cdot\cdot,X_{p+q})\\
=&\sum_{\substack{\sigma\in sh(p,q),\\\sigma(p+q)=p+q}}(-1)^{pq}sgn(\sigma)
P_{II}(X_{\sigma(1)},\cdot\cdot\cdot,X_{\sigma(p)},Q_{I}(X_{\sigma(p+1)},\cdot\cdot\cdot,X_{\sigma(p+q)})
\\&+\sum_{\substack{k=1,\\\sigma\in sh(k-1,q)}}^{p}(-1)^{q(k-1)}sgn(\sigma)
P_{I}(X_{\sigma(1)},\cdot\cdot\cdot,X_{\sigma(k-1)},x_{q+k}\wedge Q_{II}(X_{\sigma(k)},\cdot\cdot\cdot,X_{\sigma(k+q-1)},y_{k+q}),
X_{k+q+1},\cdot\cdot\cdot,X_{p+q})
\\&+\sum_{\substack{k=1,\\\sigma\in sh(k-1,q)}}^{p}(-1)^{q(k-1)}sgn(\sigma)
P_{I}(X_{\sigma(1)},\cdot\cdot\cdot,X_{\sigma(k-1)},Q_{II}(X_{\sigma(k)},\cdot\cdot\cdot,X_{\sigma(k+q-1)},x_{k+q})
\wedge y_{q+k},X_{k+q+1},\cdot\cdot\cdot,X_{p+q}),
\end{align*}
and
\begin{align*}&(P\circ Q)_{II}(X_1,\cdot\cdot\cdot,X_{p+q},z)\\
=&\sum_{\sigma\in sh(p,q)}(-1)^{pq}sgn(\sigma)P_{II}(X_{\sigma(1)},\cdot\cdot\cdot,X_{\sigma(p)},Q_{II}(X_{\sigma(p+1)},\cdot\cdot\cdot,X_{\sigma(p+q)},z)
\\&+\sum_{\substack{k=1,\\\sigma\in sh(k-1,q)}}^{p}(-1)^{q(k-1)}sgn(\sigma)
P_{II}(X_{\sigma(1)},\cdot\cdot\cdot,X_{\sigma(k-1)},x_{q+k}\wedge Q_{II}(X_{\sigma(k)},\cdot\cdot\cdot,X_{\sigma(k+q-1)},y_{k+q}),
X_{k+q+1},\cdot\cdot\cdot,X_{p+q},z)
\\&+\sum_{\substack{k=1,\\\sigma\in sh(k-1,q)}}^{p}(-1)^{q(k-1)} sgn(\sigma)
P_{II}(X_{\sigma(1)},\cdot\cdot\cdot,X_{\sigma(k-1)},Q_{II}(X_{\sigma(k)},\cdot\cdot\cdot,X_{\sigma(k+q-1)},x_{k+q})
\wedge y_{q+k},X_{k+q+1},\cdot\cdot\cdot,X_{p+q},z).
\end{align*}
In particular, for $f\in \mathcal{C}^{0}(\mathfrak g,\mathfrak g)=\mathrm{Hom}(\mathfrak g,\mathfrak g)$
and $P=(P_I,P_{II})\in \mathcal{C}^{p}(\mathfrak g,\mathfrak g)$, define
\begin{align}(P\circ f)_{I}(X_1,\cdot\cdot\cdot,X_{p})
=&\nonumber \sum_{k=1}^{p}P_{I}(X_{1},\cdot\cdot\cdot,X_{k-1},x_{k}\wedge f(y_k),X_{k+1},\cdot\cdot\cdot,X_{p})
\\&+\label{DLA1}\sum_{k=1}^{p}P_{I}(X_{1},\cdot\cdot\cdot,X_{k-1},f(x_{k})\wedge y_k,X_{k+1},\cdot\cdot\cdot,X_{p}),
\end{align}
\begin{equation}\label{DLA2}(f\circ P)_{I}(X_1,\cdot\cdot\cdot,X_{p})=f(P_{I}(X_{1},\cdot\cdot\cdot,X_{p})),
\end{equation}
\begin{align}(P\circ f)_{II}(X_1,\cdot\cdot\cdot,X_{p},z)\nonumber=&\sum_{k=1}^{p}P_{II}(X_{1},\cdot\cdot\cdot,X_{k-1},x_{k}\wedge f(y_k),X_{k+1},\cdot\cdot\cdot,X_{p},z)
\\&\label{DLA3}+\sum_{k=1}^{p}P_{I}(X_{1},\cdot\cdot\cdot,X_{k-1},f(x_{k})\wedge y_k,X_{k+1},\cdot\cdot\cdot,X_{p},z),
\end{align}
\begin{equation}\label{DLA4}(f\circ P)_{II}(X_1,\cdot\cdot\cdot,X_{p},z)=f(P_{II}(X_{1},\cdot\cdot\cdot,X_{p},z)).
\end{equation}
In order to derive sufficient and necessary conditions of Lie-Yamaguti algebras in
terms of Maurer-Cartan element of some graded Lie algebras, we define $P\bullet Q$ as follows:
for all $P=(P_I,P_{II})\in \mathcal{C}^{n}(\mathfrak g,\mathfrak g),~Q=(Q_I,Q_{II})\in \mathcal{C}^{n}(\mathfrak g,\mathfrak g)$,
\begin{equation}\label{DLY}
	P\bullet Q=\left\{
	\begin{aligned}
		&0,&p,q\neq1,\\
		&((P\bullet Q)_{I},(P\bullet Q)_{II}),&p=q=1,
	\end{aligned}
	\right.
\end{equation}
where
\begin{align*}&(P\bullet Q)_{I}(x_1,x_2,x_{3})\\
=&\frac{1}{2}\sum_{\sigma\in S_{3}}sgn(\sigma)
P_{I}(Q_{I}(x_{\sigma(1)},x_{\sigma(2)}),x_{\sigma(3)})
+\frac{1}{2}\sum_{\sigma\in S_{3}}sgn(\sigma)Q_{II}(x_{\sigma(1)},x_{\sigma(2)},x_{\sigma(3)}),
\end{align*}
\begin{align*}(P\bullet Q)_{II}(x_1,x_2,x_{3},x_{4})=\frac{1}{2}\sum_{\sigma\in S_{3}}sgn(\sigma)
P_{II}(Q_{I}(x_{\sigma(1)},x_{\sigma(2)}),x_{\sigma(3)},x_4).
\end{align*}

Let $\mathfrak g$ and $V$ be vector spaces. For any
$(\chi,\omega),(\chi',\omega')\in\mathcal{C}^{1}(\mathfrak g,\mathfrak g)$, put $\Pi=(\chi,\omega),~\Pi'=(\chi',\omega')$
and $\Pi+\Pi'=(\chi+\chi',\omega+\omega')$.

\begin{pro} \label{pro:Dga1} With the above notations, $(\mathcal{L}^{*}(\mathfrak g,\mathfrak g),[ \ , \
]_{LY})$ is a graded Lie algebra, where
\begin{equation}\label{DLY}
	[P,Q]_{LY}=\left\{
	\begin{aligned}
		&P\bullet Q+Q\bullet P+P\circ Q+Q\circ P,&p=q=1,\\
		&P\circ Q-(-1)^{pq}Q\circ P,&otherwise,
	\end{aligned}
	\right.
\end{equation}
for all $P\in \mathcal{C}^{p}(\mathfrak g,\mathfrak g),Q\in \mathcal{C}^{q}(\mathfrak g,\mathfrak g).$
Furthermore, $\Pi=(\chi,\omega)\in \mathcal{C}^{1}(\mathfrak
g,\mathfrak g)$ defines a Lie-Yamaguti algebra structure on $\mathfrak
g$ if and only if $[ \ , \ ]_{LY}=0$, that is, $\Pi$ is a
Maurer-Cartan element of the graded Lie algebra $(\mathcal{L}^{*}(\mathfrak
g,\mathfrak g),[ \ , \ ]_{LY})$. We write $[ P , Q]_{LY}=([P,Q]_{I},[P,Q]_{II}).$
\end{pro}

\begin{proof} Take the same procedure of the proof of Proposition 4.1 \cite{8}, we can check that Eq.~(\ref{Co1}) holds.
Clearly, Eq.~(\ref{Co1}) implies that Eq.~(\ref{eq2.1}) holds.
For all $\Pi=(\chi,\omega)\in \mathcal{C}^{1}(\mathfrak g,\mathfrak g)$, $[ \Pi,\Pi ]_{LY}=2\Pi\circ \Pi+2\Pi\bullet \Pi$
and for all $x_1,x_2,x_3,x_4\in \mathfrak g$, we have
\begin{align*}(\Pi\bullet \Pi)_{I}(x_1,x_2,x_3)=&\chi(\chi(x_1,x_2),x_3)+\chi(\chi(x_2,x_3),x_1)+\chi(\chi(x_3,x_1),x_2)
\\&+\omega (x_1,x_2,x_3)+\omega(x_2,x_3,x_1)+\omega(x_3,x_1,x_2),
\end{align*}
and
\begin{align*}&(\Pi\bullet \Pi)_{II}(x_1,x_2,x_3,x_4)\\=&
\omega (\chi(x_1,x_2),x_3,x_4)+\omega(\chi(x_2,x_3),x_1,x_4)+\omega(\chi(x_3,x_1),x_2,x_4).
\end{align*}
Combining Theorem 3.1 \cite{8}, $\Pi=(\chi,\omega)\in \mathcal{C}^{1}(\mathfrak
g,\mathfrak g)$ defines a Lie-Yamaguti algebra structure on $\mathfrak
g$ if and only if $\Pi$ is a
Maurer-Cartan element of the graded Lie algebra $(\mathcal{L}^{*}(\mathfrak
g,\mathfrak g),[ \ , \ ]_{LY})$.
\end{proof}

By Proposition \ref{pro:Dga1}, we rewrite Theorem 3.3 \cite{8} as follows:
\begin{thm} \label{pro:Dga2} Let $(\mathfrak g,\chi_{\mathfrak g},\omega_{\mathfrak g})$ be a Lie-Yamaguti algebra. Then
 $(\mathcal{L}^{*}(\mathfrak
g,\mathfrak g),[ \ , \ ]_{LY},d_{\Pi})$ is a differential graded
Lie algebra, where $d_{\Pi}$ with $\Pi=(\chi_{\mathfrak g},\omega_{\mathfrak g})$ is given by
\begin{equation}\label{MC5}d_{\Pi}(\nu)=[\Pi,\nu]_{LY},~~\forall~\nu\in \mathcal{C}^{n-1}(\mathfrak
g,\mathfrak g).\end{equation}
Moreover, $\Pi+\Pi'$ with $\Pi'\in \mathcal{C}^{1}(\mathfrak
g,\mathfrak g)$ defines a Lie-Yamaguti algebra structure on $\mathfrak
g$ if and only if $\Pi'$ is a Maurer-Cartan element of the
differential graded Lie algebra $(\mathcal{L}^{*}(\mathfrak
g,\mathfrak g),[ \ , \ ]_{LY},d_{\Pi})$.
\end{thm}
Denote
\begin{equation*}\bar{\mu}(x+a,y+b)=\mu(x)b-\mu(y)a\end{equation*}
and
 \begin{equation*}\bar{\theta}(x+a,y+b,z+c)=D(x,y)c+\theta(y,z)a-\theta(x,z)b\end{equation*}
 for all $x,y,z\in \mathfrak g$ and $a,b,c\in V$.

\begin{pro} \label{pro:Dga3} With the above notations, $(V,\mu,\theta,D)$ is a representation of Lie-Yamaguti algebra
 $(\mathfrak g,\chi_{\mathfrak g},\omega_{\mathfrak g})$ if and only if $\bar{\Pi}\in\mathcal{L}^{*}(\mathfrak
g\ltimes V,\mathfrak g\ltimes V)$ is a Maurer-Cartan element of the differential graded Lie
algebra $(\mathcal{L}^{*}(\mathfrak
g\ltimes V,\mathfrak g\ltimes V),[ \ , \ ]_{LY},d_{\bar{\Pi}})$, where $\bar{\Pi}=(\bar{\mu},\bar{\theta}).$
\end{pro}

\begin{proof} It follows from Theorem \ref{pro:Dga2}.
\end{proof}

Let $(\mathfrak g,\chi_{\mathfrak g},\omega_{\mathfrak g})$
and $(\mathfrak h, \chi_{\mathfrak h},\omega_{\mathfrak h})$ be two Lie-Yamaguti algebras. Then $(\mathfrak
g\oplus \mathfrak h,\chi_{\mathfrak
g\oplus \mathfrak h},\omega_{\mathfrak
g\oplus \mathfrak h})$ is a Lie-Yamaguti algebra, where $\chi_{\mathfrak
g\oplus \mathfrak h},\omega_{\mathfrak
g\oplus \mathfrak h}$ are defined by
\begin{equation*}\chi_{\mathfrak
g\oplus \mathfrak h}(x + a, y +
b) =\chi_{\mathfrak g}(x, y) +\chi_{\mathfrak h}(a,b),~~\omega_{\mathfrak
g\oplus \mathfrak h}(x + a, y +
b, z + c) =\omega_{\mathfrak g}(x, y, z) +\omega_{\mathfrak h}(a,b,c)\end{equation*}
 for all $x,y,z\in \mathfrak g,a,b,c\in \mathfrak h$.

In view of Theorem \ref{pro:Dga2}, $(\mathcal{L}^{*}(\mathfrak g\oplus
\mathfrak h,\mathfrak g\oplus \mathfrak h),[ \ , \
]_{LY},d_{(\chi_{\mathfrak
g\oplus \mathfrak h},\omega_{\mathfrak
g\oplus \mathfrak h})})$ is a differential
graded Lie algebra.
Define $\mathcal{C}_{>}^{n}(\mathfrak g\oplus \mathfrak
h,\mathfrak h)\subset \mathcal{C}^{n}(\mathfrak g\oplus \mathfrak h,\mathfrak
h),~\mathbb{C}_{>}(\mathfrak g\oplus \mathfrak
h,\mathfrak h)\subset \mathbb{C}(\mathfrak g\oplus \mathfrak h,\mathfrak
h)$ respectively by
$$\mathcal{C}^{n}(\mathfrak g\oplus \mathfrak h,\mathfrak
h)=\mathcal{C}_{>}^{n}(\mathfrak g\oplus \mathfrak h,\mathfrak h)\oplus \mathcal{C}^{n}(
\mathfrak h,\mathfrak h),~
\mathbb{C}(\mathfrak g\oplus \mathfrak h,\mathfrak h)=\mathbb{C}_{>}(\mathfrak g\oplus \mathfrak h,\mathfrak h)\oplus \mathbb{C}(
\mathfrak h,\mathfrak h).$$ Denote by $\mathcal{C}_{>}(\mathfrak g\oplus
\mathfrak h,\mathfrak h)=\oplus_{n}\mathcal{C}_{>}^{n}(\mathfrak g\oplus
\mathfrak h,\mathfrak h)$ and $\mathcal{L}_{>}(\mathfrak g\oplus
\mathfrak h,\mathfrak h)=\mathcal{C}_{>}(\mathfrak g\oplus
\mathfrak h,\mathfrak h)\oplus \mathbb{C}_{>}(\mathfrak g\oplus \mathfrak h,\mathfrak h)$.

Similar to the case of $3$-Lie algebras \cite{050}, we have

\begin{pro} With the above notations, $(\mathcal{L}_{>}(\mathfrak g\oplus
\mathfrak h,\mathfrak h),[ \ , \ ]_{LY},d_{(\chi_{\mathfrak
g\oplus \mathfrak h},\omega_{\mathfrak
g\oplus \mathfrak h})})$ is a differential graded Lie subalgebra of
$(\mathcal{L}^{*}(\mathfrak g\oplus \mathfrak h,\mathfrak g\oplus \mathfrak
h),[ \ , \ ]_{LY},d_{(\chi_{\mathfrak
g\oplus \mathfrak h},\omega_{\mathfrak
g\oplus \mathfrak h})})$.
\end{pro}

\begin{pro}
The following conditions are equivalent:
\begin{enumerate}[label=$(\roman*)$,leftmargin=15pt]
    \item $(\mathfrak g\oplus \mathfrak
h, [ \ ,  \ ]_{\chi}, \{ \ , \ , \ \}_{\omega} )$ is a Lie-Yamaguti algebra,
which is a non-abelian extension of $\mathfrak g$ by $\mathfrak h$.
    \item $\Pi=(\bar{\chi},\bar{\omega})$ is a Maurer-Cartan element of the
    differential graded Lie algebra $(C_{>}(\mathfrak g\oplus \mathfrak
h, \mathfrak h), [ \ ,  \ ]_{LY},d_{(\chi_{\mathfrak
g\oplus \mathfrak h},\omega_{\mathfrak
g\oplus \mathfrak h})})$, where
 \begin{equation*}\bar{\chi}(x+a,y+b)=\chi(x,y)+\mu(x)b-\mu(y)a,\end{equation*}
\begin{align*}\bar{\omega}(x+a,y+b,z+c)=&\omega(x,y,z)+D(x,y)c+\theta(y,z)a-\theta(x,z)b
\\&+T(z)(a,b)+\rho(x)(b,c)-\rho(y)(a,c),\end{align*}
for all $x,y,z\in \mathfrak g,a,b,c\in \mathfrak h$.
\end{enumerate}
\end{pro}

\begin{proof} In view of the definition of Maurer-Cartan element,
$\Pi=(\bar{\chi},\bar{\omega})$ is a Maurer-Cartan element of the
differential graded Lie algebra $(\mathcal{L}_{>}(\mathfrak g\oplus \mathfrak
h, \mathfrak h), [ \ ,  \ ]_{LY},d_{(\chi_{\mathfrak
g\oplus \mathfrak h},\omega_{\mathfrak
g\oplus \mathfrak h})})$ if and only if
$$d_{(\chi_{\mathfrak
g\oplus \mathfrak h},\omega_{\mathfrak
g\oplus \mathfrak h})}\Pi +\frac{1}{2}[\Pi,\Pi]_{LY}=0,$$that is,
\begin{equation}\label{Mc1}[(\chi_{\mathfrak
g\oplus \mathfrak h},\omega_{\mathfrak g\oplus \mathfrak h}),(\bar{\chi},\bar{\omega})]_{LY}+\frac{1}{2}[(\bar{\chi},\bar{\omega}),(\bar{\chi},\bar{\omega})]_{LY}=0.\end{equation}
On the other hand, by Proposition \ref{pro:Dga1}, we know that $(\mathfrak g\oplus \mathfrak
h, [ \ ,  \ ]_{\chi}, \{ \ , \ , \ \}_{\omega} )$ is a Lie-Yamaguti algebra if and only if
\begin{equation}\label{Mc2}[(\chi_{\mathfrak
g\oplus \mathfrak h}+\bar{\chi},\omega_{\mathfrak g\oplus \mathfrak h}+\bar{\omega}),(\chi_{\mathfrak
g\oplus \mathfrak h}+\bar{\chi},\omega_{\mathfrak g\oplus \mathfrak h}+\bar{\omega})]_{LY}=0.
\end{equation}
Since $(\mathfrak g\oplus \mathfrak h,\chi_{\mathfrak
g\oplus \mathfrak h},\omega_{\mathfrak g\oplus \mathfrak h})$ is a Lie-Yamaguti algebra,
by computations, we have
\begin{align*}&[(\chi_{\mathfrak
g\oplus \mathfrak h}+\bar{\chi},\omega_{\mathfrak g\oplus \mathfrak h}+\bar{\omega}),(\chi_{\mathfrak
g\oplus \mathfrak h}+\bar{\chi},\omega_{\mathfrak g\oplus \mathfrak h}+\bar{\omega})]_{LY}
\\=&[(\chi_{\mathfrak
g\oplus \mathfrak h},\omega_{\mathfrak g\oplus \mathfrak h}),(\chi_{\mathfrak
g\oplus \mathfrak h},\omega_{\mathfrak g\oplus \mathfrak h})]_{LY}
+[(\bar{\chi},\bar{\omega}),(\bar{\chi},\bar{\omega})]_{LY}
+2[(\chi_{\mathfrak
g\oplus \mathfrak h},\omega_{\mathfrak g\oplus \mathfrak h}),(\bar{\chi},\bar{\omega})]_{LY}
\\=&
[(\bar{\chi},\bar{\omega}),(\bar{\chi},\bar{\omega})]_{LY}
+2[(\chi_{\mathfrak
g\oplus \mathfrak h},\omega_{\mathfrak g\oplus \mathfrak h}),(\bar{\chi},\bar{\omega})]_{LY},
\end{align*}
which yields that Eq.~(\ref{Mc1}) holds if and only if Eq.~(\ref{Mc2}) holds. This completes the proof.
\end{proof}

\begin{pro}
Two non-abelian extensions $(\mathfrak g\oplus \mathfrak h, [ \ ,  \ ]_{\chi_0},\{ \ , \ , \ \}_{\omega_0} )$ and
$(\mathfrak g\oplus \mathfrak h, [ \ ,  \ ]_{\chi},\{ \ , \ , \ \}_{\omega} )$ are equivalent if and only if
the Maurer-Cartan elements $\Pi_{0}=(\bar{\chi}_{0},\bar{\omega}_{0})$ and
 $\Pi=(\bar{\chi},\bar{\omega})$  are gauge equivalent.
\end{pro}

\begin{proof}
Let $\Pi_{0},\Pi$ be
two Maurer-Cartan elements of the differential graded Lie algebra
$(\mathcal{L}_{>}(\mathfrak g\oplus \mathfrak h, \mathfrak h), [ \ , \
]_{LY},d_{(\chi_{\mathfrak g\oplus \mathfrak h},\omega_{\mathfrak g\oplus \mathfrak h})})$. $\Pi_{0},\Pi$ are
gauge equivalent if and only if there is a linear map $\varphi\in \mathrm{Hom} (\mathfrak
g, \mathfrak h)$ such that
\begin{align*}\Pi_{0}=&e^{ad_{\varphi}}\Pi
-\frac{e^{ad_{\varphi}}-1}{ad_{\varphi}}d_{(\chi_{\mathfrak g\oplus\mathfrak h},\omega_{\mathfrak g\oplus\mathfrak h})}\varphi
\\=&(id+ad_{\varphi}+\frac{1}{2!}ad_{\varphi}^{2}+\cdot\cdot\cdot++\frac{1}{n!}ad_{\varphi}^{n}+\cdot\cdot\cdot)\Pi
\\&-(id+\frac{1}{2!}ad_{\varphi}+\frac{1}{3!}ad_{\varphi}^{2}+\cdot\cdot\cdot+\frac{1}{n!}ad_{\varphi}^{n-1}+\cdot\cdot\cdot)
d_{(\chi_{\mathfrak g\oplus\mathfrak h},\omega_{\mathfrak g\oplus\mathfrak h})}\varphi.\end{align*}
In the following, we denote by
\begin{equation*}ad_{\varphi}\Pi=[\varphi,\Pi]_{LY}=([\varphi,\Pi]_{I},[\varphi,\Pi]_{II}),~~ad_{\varphi}^{2} \Pi =[\varphi,[\varphi,\Pi]_{LY}]_{LY}=([\varphi,[\varphi,\Pi]_{LY}]_{I},[\varphi,[\varphi,\Pi]_{LY}]_{II}),\end{equation*}
\begin{equation*}d_{(\chi_{\mathfrak g\oplus \mathfrak h},\omega_{\mathfrak g\oplus \mathfrak h})}\varphi
=[(\chi_{\mathfrak g\oplus \mathfrak h},\omega_{\mathfrak g\oplus \mathfrak h}),\varphi]_{LY}
=([(\chi_{\mathfrak g\oplus \mathfrak h},\omega_{\mathfrak g\oplus \mathfrak h}),\varphi]_{I},
[(\chi_{\mathfrak g\oplus \mathfrak h},\omega_{\mathfrak g\oplus \mathfrak h}),\varphi]_{II}),\end{equation*}
and
\begin{equation*}[\varphi,d_{(\chi_{\mathfrak g\oplus \mathfrak h},\omega_{\mathfrak g\oplus \mathfrak h})}(\varphi)]_{LY}
=([\varphi,[(\chi_{\mathfrak g\oplus \mathfrak h},\omega_{\mathfrak g\oplus \mathfrak h}),\varphi]_{LY}]_{I},
[\varphi,[(\chi_{\mathfrak g\oplus \mathfrak h},\omega_{\mathfrak g\oplus \mathfrak h}),\varphi]_{LY}]_{II}).
\end{equation*}
Using Eqs.~(\ref{DLA1})-(\ref{DLA4}), for all $w_i=x_i+a_i\in \mathfrak g\oplus\mathfrak h~(i=1,2,3)$,
 we get
\begin{align*}[\varphi,\Pi]_{I}(w_1,w_2)=&\varphi\bar{\chi}(w_1,w_2)-\bar{\chi}(w_1,\varphi(x_2))
-\bar{\chi}(\varphi(x_1),w_2)
\\=&-\mu(x_1)\varphi(x_2)+\mu(x_2)\varphi(x_1),
\end{align*}
\begin{align*}[\varphi,\Pi]_{II}(w_1,w_2,w_3)=&\varphi\bar{\omega}(w_1,w_2,w_3)
-\bar{\omega}(\varphi(x_1),w_2,w_3)-\bar{\omega}(w_1,w_2,\varphi(x_3))-\bar{\omega}(w_1,\varphi(x_2),w_3)
\\=&\theta(x_1,x_3)\varphi(x_2)-T(x_3)(a_1,\varphi(x_2))-\rho(x_1)(\varphi(x_2),a_3)\\&
   -\theta(x_2,x_3)\varphi(x_1)-T(x_3)(\varphi(x_1),a_2)+\rho(x_2)(\varphi(x_1),a_3)
\\&+\rho(x_2)(a_1,\varphi(x_3))-D(x_1,x_2)\varphi(x_3)-\rho(x_1)(a_2,\varphi(x_3)),
\end{align*}
\begin{equation*}[\varphi,[\varphi,\Pi]_{LY}]_{I}(w_1,w_2)=\varphi[\varphi,\Pi]_{I}(w_1,w_2)
-[\varphi,\chi]_{I}(w_1,\varphi(x_2))-[\varphi,\Pi]_{I}(\varphi(x_1),w_2) =0,
\end{equation*}
\begin{align*}[\varphi,[\varphi,\Pi]_{LY}]_{II}(w_1,w_2,w_3)=&\varphi[\varphi,\Pi]_{II}(w_1,w_2,w_3)
-[\varphi,\Pi]_{II}(w_1,\varphi(x_2),w_3)
\\&-[\varphi,\Pi]_{II}(\varphi(x_1),w_2,w_3)
-[\varphi,\Pi]_{II}(w_1,w_2,\varphi(x_3))\\=&
2T(x_3)(\varphi(x_1),\varphi(x_2))+2\rho(x_1)(\varphi(x_2),\varphi(x_3))-2\rho(x_2)(\varphi(x_1),\varphi(x_3)),
\end{align*}
\begin{align*}[(\chi_{\mathfrak g\oplus \mathfrak h},\omega_{\mathfrak g\oplus \mathfrak h}),\varphi]_{I}(w_1,w_2)
=&\chi_{\mathfrak g\oplus \mathfrak h}(w_1,\varphi(x_2))+\chi_{\mathfrak g\oplus \mathfrak h}(\varphi(x_1),w_2)
-\varphi\chi_{\mathfrak g\oplus \mathfrak h}(w_1,w_2)
\\=&[a_1,\varphi(x_2)]_{\mathfrak h}+[\varphi(x_1),a_2]_{\mathfrak h}-\varphi([x_1,x_2]_{\mathfrak g}),
\end{align*}
\begin{align*}&[(\chi_{\mathfrak g\oplus \mathfrak h},\omega_{\mathfrak g\oplus \mathfrak h}),\varphi]_{II}(w_1,w_2,w_3)
\\=&\omega_{\mathfrak g\oplus \mathfrak h}(w_1,\varphi(x_2),w_3)
+\omega_{\mathfrak g\oplus \mathfrak h}(\varphi(x_1),w_2,w_3)
+\omega_{\mathfrak g\oplus \mathfrak h}(w_1,w_2,\varphi(x_3))-\varphi\omega_{\mathfrak g\oplus \mathfrak h}(w_1,w_2,w_3)
\\=&\{a_1,\varphi(x_2),a_3\}_{\mathfrak h}+
\{\varphi(x_1),a_2,a_3\}_{\mathfrak h}+\{a_1,a_2,\varphi(x_3)\}_{\mathfrak h}-\varphi(\{x_1,x_2,x_3\}_{\mathfrak g}),
\\&[\varphi,d_{(\chi_{\mathfrak g\oplus \mathfrak h},\omega_{\mathfrak g\oplus \mathfrak h})}(\varphi)]_{I}(w_1,w_2)
\\=&\varphi[(\chi_{\mathfrak g\oplus \mathfrak h},\omega_{\mathfrak g\oplus \mathfrak h}),\varphi]_{I}(w_1,w_2)-
[(\chi_{\mathfrak g\oplus \mathfrak h},\omega_{\mathfrak g\oplus \mathfrak h}),\varphi]_{I}(w_1,\varphi(x_2))
-[(\chi_{\mathfrak g\oplus \mathfrak h},\omega_{\mathfrak g\oplus \mathfrak h}),\varphi]_{I}(\varphi(x_1),w_2)
\\=&-2[\varphi(x_1),\varphi(x_2)]_{\mathfrak h},
\\&[\varphi,d_{(\chi_{\mathfrak g\oplus \mathfrak h},\omega_{\mathfrak g\oplus \mathfrak h})}(\varphi)]_{II}(w_1,w_2,w_3)
\\=&\varphi[(\chi_{\mathfrak g\oplus \mathfrak h},\omega_{\mathfrak g\oplus \mathfrak h}),\varphi]_{II}(w_1,w_2,w_3)
-[(\chi_{\mathfrak g\oplus \mathfrak h},\omega_{\mathfrak g\oplus \mathfrak h}),\varphi]_{II}(w_1,\varphi(x_2),w_3)
\\&-[(\chi_{\mathfrak g\oplus \mathfrak h},\omega_{\mathfrak g\oplus \mathfrak h}),\varphi]_{II}(\varphi(x_1),w_2,w_3)
-[(\chi_{\mathfrak g\oplus \mathfrak h},\omega_{\mathfrak g\oplus \mathfrak h}),\varphi]_{II}(w_1,w_2,\varphi(x_3))
\\=&-\{\varphi(x_1),\varphi(x_2),a_3\}_{\mathfrak h}-\{a_1,\varphi(x_2),\varphi(x_3)\}_{\mathfrak h}
-\{\varphi(x_1),\varphi(x_2),a_3\}_{\mathfrak h}\\&-\{\varphi(x_1),a_2,\varphi(x_3)\}_{\mathfrak h}
-\{a_1,\varphi(x_2),\varphi(x_3)\}_{\mathfrak h}-\{\varphi(x_1),a_2,\varphi(x_3)\}_{\mathfrak h}
\\=&-2\{\varphi(x_1),\varphi(x_2),a_3\}_{\mathfrak h}-2\{a_1,\varphi(x_2),\varphi(x_3)\}_{\mathfrak h}
-2\{\varphi(x_1),a_2,\varphi(x_3)\}_{\mathfrak h},
\\&[\varphi,[\varphi,d_{(\chi_{\mathfrak g\oplus \mathfrak h},\omega_{\mathfrak g\oplus \mathfrak h})}(\varphi)]_{LY}]_{I}(w_1,w_2)
\\=&\varphi[\varphi,d_{(\chi_{\mathfrak g\oplus \mathfrak h},\omega_{\mathfrak g\oplus \mathfrak h})}(\varphi)]_{I}(w_1,w_2)
-[\varphi,d_{(\chi_{\mathfrak g\oplus \mathfrak h},\omega_{\mathfrak g\oplus \mathfrak h})}(\varphi)]_{I}(x_1+a_1,\varphi(x_2))
-[\varphi,d_{(\chi_{\mathfrak g\oplus \mathfrak h},\omega_{\mathfrak g\oplus \mathfrak h})}(\varphi)]_{I}(\varphi(x_1),w_2)
\\=&0,
\\&[\varphi,[\varphi,d_{(\chi_{\mathfrak g\oplus \mathfrak h},
\omega_{\mathfrak g\oplus \mathfrak h})}(\varphi)]_{LY}]_{II}(w_1,w_2,w_3)
\\=&[\varphi,d_{(\chi_{\mathfrak g\oplus \mathfrak h},\omega_{\mathfrak g\oplus \mathfrak h})}(\varphi)]_{II}(x_1+a_1,\varphi(x_2),w_3)
+[\varphi,d_{(\chi_{\mathfrak g\oplus \mathfrak h},\omega_{\mathfrak g\oplus \mathfrak h})}(\varphi)]_{II}(\varphi(x_1),w_2,w_3)
\\&+[\varphi,d_{(\chi_{\mathfrak g\oplus \mathfrak h},\omega_{\mathfrak g\oplus \mathfrak h})}(\varphi)]_{II}(w_1,w_2,\varphi(x_3))
-\varphi[\varphi,d_{(\chi_{\mathfrak g\oplus \mathfrak h},\omega_{\mathfrak g\oplus \mathfrak h})}(\varphi)]_{II}(w_1,w_2,w_3)
\\=&6\{\varphi(x_1),\varphi(x_2),\varphi(x_3)\}_{\mathfrak h},
\end{align*}
and
$$ad_{\varphi}^{n}\Pi=0,~~ad_{\varphi}^{n}(d_{(\chi_{\mathfrak g\oplus \mathfrak h},\omega_{\mathfrak g\oplus \mathfrak h})}\Pi)=0,~~\forall~~n\geq 3.$$
Thus, $\Pi$ and $\Pi_{0}$ are gauge equivalent Maurer-Cartan elements if and only if
\begin{equation}\label{GMC1}\chi_{0}=\chi+[\varphi,\Pi]_{I}
-d_{(\chi_{\mathfrak g\oplus \mathfrak h},\omega_{\mathfrak g\oplus \mathfrak h})}\varphi
-\frac{1}{2!}[\varphi,d_{(\chi_{\mathfrak g\oplus \mathfrak h},\omega_{\mathfrak g\oplus \mathfrak h})}\varphi]_{I},
\end{equation}
\begin{align}\omega_{0}\nonumber=&\omega+[\varphi,\Pi]_{II}+\frac{1}{2!}[\varphi,[\varphi,\Pi]_{LY}]_{II}
-d_{(\chi_{\mathfrak g\oplus \mathfrak h},\omega_{\mathfrak g\oplus \mathfrak h})}\varphi
\\&\label{GMC2}-\frac{1}{2!}[\varphi,d_{(\chi_{\mathfrak g\oplus \mathfrak h},\omega_{\mathfrak g\oplus \mathfrak h})}\varphi]_{II}
-\frac{1}{3!}[\varphi,[\varphi,d_{(\chi_{\mathfrak g\oplus \mathfrak h},\omega_{\mathfrak g\oplus \mathfrak h})}\varphi]_{LY}]_{II}.
\end{align}
Therefore, Eqs.~ (\ref{GMC1}) and (\ref{GMC2}) hold if and
only if Eqs.~(\ref{E1})-(\ref{E6}) hold. This finishes the proof.

\end{proof}

%%%%%%%%%%%%%%%%%%%%%%%%%%%%%%%%%%%%%%%%%%%%%%%%%%%%%%%%%%%%%%%%%%%%%%%%%%%%%%%%%%%%%%%%%%%%%%%%%%%%%%%
\section{Extensibility of a pair of Lie-Yamaguti algebra automorphisms}
In this section, we study extensibility of pairs of Lie-Yamaguti algebra
automorphisms and characterize them by equivalent conditions.

Let $\mathcal{E}:0\longrightarrow\mathfrak h\stackrel{i}{\longrightarrow} \hat{\mathfrak g}\stackrel{p}{\longrightarrow}\mathfrak g\longrightarrow0$
  be a non-abelian extension of $\mathfrak g$ by $\mathfrak h$ with a section $s$ of $p$.
Denote  $\mathrm{Aut}_{\mathfrak h}(\hat{\mathfrak g})=\{\gamma\in \mathrm{Aut} (\hat{\mathfrak g})\mid \gamma(\mathfrak h)=\mathfrak h\}.$

\begin{defi} A pair of automorphisms $(\alpha,\beta)\in \mathrm{Aut} (\mathfrak g)\times \mathrm{Aut} (\mathfrak h)$
is said to be extensible with respect to a non-abelian extension
 $$\mathcal{E}:0\longrightarrow \mathfrak h\stackrel{i}{\longrightarrow} \hat{\mathfrak g}\stackrel{p}{\longrightarrow}\mathfrak g\longrightarrow0$$
if there is an automorphism $\gamma\in \mathrm{Aut}_{\mathfrak h}
(\hat{\mathfrak g})$ such that $i\beta=\gamma i,~p\gamma=\alpha p$, that is, the following commutative diagram holds:
\[\xymatrix@C=20pt@R=20pt{0\ar[r]&\mathfrak h\ar[d]_\beta\ar[r]^i&\hat{\mathfrak g}\ar[d]_\gamma\ar[r]^p&\mathfrak g
\ar[d]_\alpha\ar[r]&0\\0\ar[r]&\mathfrak h\ar[r]^i&\hat{\mathfrak g}\ar[r]^p&\mathfrak g\ar[r]&0.}\]
\end{defi}
It is natural to ask: when is a pair of automorphisms $(\alpha,\beta)\in \mathrm{Aut} (\mathfrak g)\times \mathrm{Aut} (\mathfrak h)$
extensible? We discuss this problem in the following.

\begin{thm} \label{EC} Let $0\longrightarrow\mathfrak h\stackrel{i}{\longrightarrow}
\hat{\mathfrak g}\stackrel{p}{\longrightarrow}\mathfrak
g\longrightarrow0$ be a non-abelian extension of $\mathfrak g$ by
$\mathfrak h$ with a section $s$ of $p$ and
$(\chi,\omega,\mu,\theta,D,\rho,T)$ the corresponding non-abelian (2,3)-cocycle
induced by $s$. A pair $(\alpha,\beta)\in \mathrm{Aut}(\mathfrak
g)\times \mathrm{Aut}(\mathfrak h)$ is extensible if and only if there is a
linear map $\varphi:\mathfrak g\longrightarrow \mathfrak h$
satisfying the following conditions:
\begin{align}\label{Iam1}
     &\nonumber\beta\omega(x,y,z)-\omega(\alpha(x),\alpha(y),\alpha(z))=T(\alpha(z))(\varphi(x),\varphi(y))-\rho(\alpha(y))(\varphi(x),\varphi(z))
-\theta(\alpha(y),\alpha(z))\varphi(x)\\+&\rho(\alpha(x))(\varphi(y),\varphi(z))+\theta(\alpha(x),\alpha(z))\varphi(y)
-D(\alpha(x),\alpha(y))\varphi(z)+\varphi(\{x,y,z\}_{\mathfrak
     h})-\{\varphi(x),\varphi(y),\varphi(z)\}_{\mathfrak h},
\end{align}
\begin{equation}\label{Iam2}
     \beta \chi(x,y)-\chi(\alpha(x),\alpha(y))=[\varphi(x),\varphi(y)]_{\mathfrak h}+\varphi([x,y]_{\mathfrak g})
     -\mu(\alpha(x))\varphi(y)+\mu(\alpha(y))\varphi(x),
\end{equation}
\begin{equation}\label{Iam3}
     \beta(\theta(x,y)a)-\theta(\alpha(x),\alpha(y))\beta(a)=\{\beta(a),\varphi(x),\varphi(y)\}_{\mathfrak
     h}-T(\alpha(y))(\beta(a),\varphi(x))+\rho(\alpha(x))(\beta(a),\varphi(y)),
\end{equation}
\begin{equation}\label{Iam4}
     \beta D(x,y)a-D(\alpha(x),\alpha(y))\beta(a)=\{\varphi(x),\varphi(y),\beta(a)\}_{\mathfrak
     h}-\rho(\alpha(x))(\varphi(y),\beta(a))+\rho(\alpha(y))(\varphi(x),\beta(a)),
\end{equation}
\begin{equation}\label{Iam5}
     \beta(\rho(x)(a,b))-\rho(\alpha(x))(\beta(a),\beta(b))=\{\beta(a),\varphi(x),\beta(b)\}_{\mathfrak h},
\end{equation}
\begin{equation}\label{Iam6}
     \beta T(x)(a,b)-T(\alpha(x))(\beta(a),\beta(b))=\{\beta(b),\beta(a),\varphi(x)\}_{\mathfrak h},
\end{equation}
\begin{equation}\label{Iam7}
     \beta \mu(x)a-\mu(\alpha(x))\beta(a)=[\beta(a),\varphi(x)]_{\mathfrak h},
\end{equation}
for all $x,y,z\in \mathfrak g$ and $a\in \mathfrak h$.
\end{thm}

\begin{proof} Assume that $(\alpha,\beta)\in \mathrm{Aut}(\mathfrak
g)\times \mathrm{Aut}(\mathfrak h)$ is extensible, that is, there is an
automorphism $\gamma\in \mathrm{Aut}_{\mathfrak h}(\hat{\mathfrak g})$ such
that $\gamma i=i\beta$ and $p\gamma =\alpha p$. Due to $s$ being a section of $p$,
for all
$x\in \mathfrak g$,
$$p(s\alpha-\gamma s)(x)=\alpha(x)-\alpha(x)=0,$$
which implies that $(s\alpha-\gamma s)(x)\in \mathrm{ker}p=\mathfrak h$.
So we can define a linear map $\varphi:\mathfrak g\longrightarrow
\mathfrak h$ by
$$\varphi(x)=(s\alpha-\gamma s)(x),~~\forall~x\in \mathfrak g.$$
 Using Eqs.~(\ref{C2}) and (\ref{C3}), for $x,y\in \mathfrak g, a\in \mathfrak h$, we get
\begin{eqnarray*}&&
      \beta(\theta(x,y)a)-\theta(\alpha(x),\alpha(y))\beta(a)
      \\&=&\beta\{a,s(x),s(y)\}_{\hat{\mathfrak g}}-\{\beta(a),s\alpha(x),s\alpha(y)\}_{\hat{\mathfrak g}}
      \\&=&\{\beta(a),\beta s(x),\beta s(y)\}_{\hat{\mathfrak g}}-\{\beta(a),s\alpha(x),s\alpha(y)\}_{\hat{\mathfrak g}}
      \\&=&\{\beta(a),\beta s(x)-s\alpha(x),\beta s(y)\}_{\hat{\mathfrak
      g}}+\{\beta(a),s\alpha(x),\beta s(y)\}_{\hat{\mathfrak g}}-\{\beta(a),s\alpha(x),s\alpha(y)\}_{\hat{\mathfrak g}}
      \\&=&\{\beta(a),-\varphi(x),\beta s(y)\}_{\hat{\mathfrak
      g}}+\{\beta(a),s\alpha(x),-\varphi(y)\}_{\hat{\mathfrak g}}
\\&=&-\{\beta(a),\varphi(x),\beta s(y)-s\alpha(y)\}_{\hat{\mathfrak
      g}}-\{\beta(a),\varphi(x),s\alpha(y)\}_{\hat{\mathfrak
      g}}- \{\beta(a),s\alpha(x),\varphi(y)\}_{\hat{\mathfrak g}}
\\&=&\{\beta(a),\varphi(x),\varphi(y)\}_{\hat{\mathfrak
      g}}-\{\beta(a),\varphi(x),s\alpha(y)\}_{\hat{\mathfrak
      g}}+ \{s\alpha(x),\beta(a),\varphi(y)\}_{\hat{\mathfrak g}}
\\&=&\{\beta(a),\varphi(x),\varphi(y)\}_{\mathfrak
     h}-T(\alpha(y))(\beta(a),\varphi(x))+\rho(\alpha(x))(\beta(a),\varphi(y)),
\end{eqnarray*}
which indicates that Eq.~(\ref{Iam3}) holds. Take the same procedure, we can prove that Eqs.~(\ref{Iam1}), (\ref{Iam2}) and
(\ref{Iam4})-(\ref{Iam7}) hold.

Conversely, suppose that $(\alpha,\beta)\in \mathrm{Aut}(\mathfrak g)\times
\mathrm{Aut}(\mathfrak h)$ and there is a linear map $\varphi:\mathfrak
g\longrightarrow \mathfrak h$ satisfying Eqs. (\ref{Iam1})-(\ref{Iam7}). Since $s$ is a section of $p$,
all $\hat{w}\in \hat{\mathfrak g}$ can be written as
$\hat{w}=a+s(x)$ for some $a\in \mathfrak h,x\in \mathfrak g.$
Define a linear map $\gamma:\hat{\mathfrak g}\longrightarrow
\hat{\mathfrak g}$ by
$$\gamma(\hat{w})=\gamma(a+s(x))=\beta(a)-\varphi(x)+s\alpha(x).$$
It is easy to check that $i\beta=\gamma i,~p\gamma=\alpha p$ and $\gamma(\mathfrak h)=\mathfrak h$.
In the sequel, firstly, we prove that $\gamma$ is bijective.
Indeed if $\gamma(\hat{w})=0,$ we have $s\alpha(x)=0$ and
$\beta(a)-\varphi(x)=0$. In view of $s$ and $\alpha$ being
injective, we get $x=0$, which follows that $a=0$. Thus,
$\hat{w}=a+s(x)=0$, that is $\gamma$ is injective. For any
$\hat{w}=a+s(x)\in \hat{\mathfrak g}$,
$$\gamma(\beta^{-1}(a)+\beta^{-1}\varphi\alpha^{-1}(x)+s\alpha^{-1}(x))=a+s(x)=\hat{w},$$
which yields that $\gamma$ is surjective. In all, $\gamma$ is
bijective.

Secondly, we show that $\gamma$ is a homomorphism of the Lie-Yamaguti algebra $\hat{\mathfrak
g}$. In fact, for all $\hat{w}_i=a_i+s(x_i)\in \hat{\mathfrak
g}~(i=1,2,3)$,
\begin{eqnarray*}&&
      \{\gamma(\hat{w}_1),\gamma(\hat{w}_2),\gamma(\hat{w}_3)\}_{\hat{\mathfrak g}}
\\&=&\{\beta(a_1)-\varphi(x_1)+s\alpha(x_1),\beta(a_2)-\varphi(x_2)+s\alpha(x_2),\beta(a_3)-\varphi(x_3)+s\alpha(x_3)\}_{\hat{\mathfrak g}}
\\&=&\{\beta(a_1),\beta(a_2),\beta(a_3)\}_{\hat{\mathfrak g}}-\{\beta(a_1),\beta(a_2),\varphi(x_3)\}_{\hat{\mathfrak
g}}+\{\beta(a_1),\beta(a_2),s\alpha(x_3)\}_{\hat{\mathfrak g}}
\\&&-\{\beta(a_1),\varphi(x_2),\beta(a_3)\}_{\hat{\mathfrak
g}}+\{\beta(a_1),\varphi(x_2),\varphi(x_3)\}_{\hat{\mathfrak
g}}-\{\beta(a_1),\varphi(x_2),s\alpha(x_3)\}_{\hat{\mathfrak g}}
\\&&+\{\beta(a_1),s\alpha(x_2),\beta(a_3)\}_{\hat{\mathfrak
g}}-\{\beta(a_1),s\alpha(x_2),\varphi(x_3)\}_{\hat{\mathfrak
g}}+\{\beta(a_1),s\alpha(x_2),s\alpha(x_3)\}_{\hat{\mathfrak g}}
\\&&-\{\varphi(x_1),\beta(a_2),\beta(a_3)\}_{\hat{\mathfrak
g}}+\{\varphi(x_1),\beta(a_2),\varphi(x_3)\}_{\hat{\mathfrak
g}}-\{\varphi(x_1),\beta(a_2),s\alpha(x_3)\}_{\hat{\mathfrak g}}
\\&&+\{\varphi(x_1),\varphi(x_2),\beta(a_3)\}_{\hat{\mathfrak
g}}-\{\varphi(x_1),\varphi(x_2),\varphi(x_3)\}_{\hat{\mathfrak
g}}+\{\varphi(x_1),\varphi(x_2),s\alpha(x_3)\}_{\hat{\mathfrak g}}
\\&&-\{\varphi(x_1),s\alpha(x_2),\beta(a_3)\}_{\hat{\mathfrak
g}}+\{\varphi(x_1),s\alpha(x_2),\varphi(x_3)\}_{\hat{\mathfrak
g}}-\{\varphi(x_1),s\alpha(x_2),s\alpha(x_3)\}_{\hat{\mathfrak g}}
\\&&+\{s\alpha(x_1),\beta(a_2),\beta(a_3)\}_{\hat{\mathfrak
g}}-\{s\alpha(x_1),\beta(a_2),\varphi(x_3)\}_{\hat{\mathfrak
g}}+\{s\alpha(x_1),\beta(a_2),s\alpha(x_3)\}_{\hat{\mathfrak g}}
\\&&-\{s\alpha(x_1),\varphi(x_2),\beta(a_3)\}_{\hat{\mathfrak
g}}+\{s\alpha(x_1),\varphi(x_2),\varphi(x_3)\}_{\hat{\mathfrak
g}}-\{s\alpha(x_1),\varphi(x_2),s\alpha(x_3)\}_{\hat{\mathfrak g}}
\\&&+\{s\alpha(x_1),s\alpha(x_2),\beta(a_3)\}_{\hat{\mathfrak
g}}-\{s\alpha(x_1),s\alpha(x_2),\varphi(x_3)\}_{\hat{\mathfrak
g}}+\{s\alpha(x_1),s\alpha(x_2),s\alpha(x_3)\}_{\hat{\mathfrak g}}
\\&=&\{\beta(a_1),\beta(a_2),\beta(a_3)\}_{\hat{\mathfrak
g}}-\{\beta(a_1),\beta(a_2),\varphi(x_3)\}_{\hat{\mathfrak
g}}+T(\alpha(x_3))(\beta(a_1),\beta(a_2))
\\&&-\{\beta(a_1),\varphi(x_2),\beta(a_3)\}_{\hat{\mathfrak
g}}+\{\beta(a_1),\varphi(x_2),\varphi(x_3)\}_{\hat{\mathfrak
g}}-T(\alpha(x_3))(\beta(a_1),\varphi(x_2))
\\&&-\rho(\alpha(x_2))(\beta(a_1),\beta(a_3))+\rho(\alpha(x_2))
(\beta(a_1),\varphi(x_3))+\theta(\alpha(x_2),\alpha(x_3))\beta(a_1)
\\&&-\{\varphi(x_1),\beta(a_2),\beta(a_3)\}_{\hat{\mathfrak
g}}+\{\varphi(x_1),\beta(a_2),\varphi(x_3)\}_{\hat{\mathfrak
g}}-T(\alpha(x_3))(\varphi(x_1),\beta(a_2))
\\&&+\{\varphi(x_1),\varphi(x_2),\beta(a_3)\}_{\hat{\mathfrak
g}}-\{\varphi(x_1),\varphi(x_2),\varphi(x_3)\}_{\hat{\mathfrak
g}}+T(\alpha(x_3))(\varphi(x_1),\varphi(x_2))
\\&&+\rho(\alpha(x_2))(\varphi(x_1),\beta(a_3))-\rho(\alpha(x_2))(\varphi(x_1),\varphi(x_3))
-\theta(\alpha(x_2),\alpha(x_3))\varphi(x_1)
\\&&+\rho(\alpha(x_1))(\beta(a_2),\beta(a_3))-\rho(\alpha(x_1))(\beta(a_2),\varphi(x_3))
-\theta(\alpha(x_1),\alpha(x_3))\beta(a_2)
\\&&-\rho(\alpha(x_1))(\varphi(x_2),\beta(a_3))+\rho(\alpha(x_1))(\varphi(x_2),\varphi(x_3))
+\theta(\alpha(x_1),\alpha(x_3))\varphi(x_2)
\\&&+D(\alpha(x_1),\alpha(x_2))\beta(a_3)-D(\alpha(x_1),\alpha(x_2))\varphi(x_3)
+\omega(\alpha(x_1),\alpha(x_2),\alpha(x_3))+s\alpha\{x_1,x_2,x_3\}_{\mathfrak g}
\end{eqnarray*}
and
\begin{eqnarray*}&&
      \gamma(\{\hat{w}_1,\hat{w}_2,\hat{w}_3\}_{\hat{\mathfrak g}})
\\&=&\gamma(\{a_1,a_2,a_3\}_{\hat{\mathfrak g}}+\{a_1,a_2,s(x_3)\}_{\hat{\mathfrak g}}+\{a_1,s(x_2),a_3\}_{\hat{\mathfrak g}}
+\{a_1,s(x_2),s(x_3)\}_{\hat{\mathfrak
g}}\\&&+\{s(x_1),a_2,s(x_3)\}_{\hat{\mathfrak
g}}+\{s(x_1),s(x_2),a_3\}_{\hat{\mathfrak
g}}+[s(x_1),a_2,a_3]_{\hat{\mathfrak g}}+\omega(x_1,x_2,x_3)+s\{x_1,x_2,x_3\}_{\mathfrak g})
\\&=&\{\beta(a_1),\beta(a_2),\beta(a_3)\}_{\hat{\mathfrak g}}+\beta( T(x_3)(a_1,a_2)-\rho(x_2)
(a_1,a_3) +\theta(x_2,x_3)a_1
-\theta(x_1,x_3)a_2
\\&&+D(x_1,x_2)a_3+\rho(x_1)(a_2,a_3))+\beta(\{ s(x_1), s(x_2),
s(x_3)\}_{\hat{\mathfrak g}})+\beta \omega(x_1,x_2,x_3)\\&&+s\alpha \{x_1,x_2,x_3\}_{\mathfrak g}-\varphi(\{x_1,x_2,x_3\}_{\mathfrak g}).
\end{eqnarray*}
Thanks to Eqs.~~(\ref{Iam1})-(\ref{Iam5}), we have
$$ \gamma(\{\hat{w}_1,\hat{w}_2,\hat{w}_3\}_{\hat{\mathfrak
g}})=\{\gamma(\hat{w}_1),\gamma(\hat{w}_2),\gamma(\hat{w}_3)\}_{\hat{\mathfrak
g}}.$$ By the same token, $ \gamma([\hat{w}_1,\hat{w}_2]_{\hat{\mathfrak
g}})=[\gamma(\hat{w}_1),\gamma(\hat{w}_2)]_{\hat{\mathfrak
g}}.$ Hence, $\gamma\in \mathrm{Aut}_{\mathfrak h}(\hat{\mathfrak g})$. This completes the proof.
\end{proof}

Let $\mathcal{E}:0\longrightarrow\mathfrak h\stackrel{i}{\longrightarrow}
\hat{\mathfrak g}\stackrel{p}{\longrightarrow}\mathfrak
g\longrightarrow0$ be a non-abelian extension of
$\mathfrak g$ by $\mathfrak h$ with a section $s$ of $p$ and $(\chi,\omega,\mu,\theta,D,\rho,T)$ be the
corresponding non-abelian (2,3)-cocycle induced by $s$.
For all $(\alpha,\beta)\in \mathrm{Aut}(\mathfrak
g)\times \mathrm{Aut}(\mathfrak h)$, define maps $\chi_{(\alpha,\beta)}:\mathfrak g\otimes
\mathfrak g\longrightarrow \mathfrak
h,~\omega_{(\alpha,\beta)}:\mathfrak g\otimes
\mathfrak g\otimes\mathfrak g\longrightarrow \mathfrak
h,~\mu_{(\alpha,\beta)}:\mathfrak g\longrightarrow \mathfrak{gl}(\mathfrak h),~\theta_{(\alpha,\beta)},
D_{(\alpha,\beta)}:\mathfrak g\wedge \mathfrak
g\longrightarrow \mathfrak{gl}(\mathfrak h),~\rho_{(\alpha,\beta)},T_{(\alpha,\beta)}:\mathfrak
g\longrightarrow \mathrm{Hom} (\mathfrak h\wedge \mathfrak h,\mathfrak h)$
 respectively by
 \begin{equation}\label{Inc1}\omega_{(\alpha,\beta)}(x,y,z)=\beta\omega(\alpha^{-1}(x),\alpha^{-1}(y),\alpha^{-1}(z)),~~
 \chi_{(\alpha,\beta)}(x,y)=\beta\chi(\alpha^{-1}(x),\alpha^{-1}(y)),\end{equation}
 \begin{equation}\label{Inc3}\theta_{(\alpha,\beta)}(x,y)a=\beta(\theta(\alpha^{-1}(x),\alpha^{-1}(y))\beta^{-1}(a))
,~~D_{(\alpha,\beta)}(x,y)a=\beta D(\alpha^{-1}(x),\alpha^{-1}(y))\beta^{-1}(a),\end{equation}
 \begin{equation}\label{Inc5}\rho_{(\alpha,\beta)}(x)(a,b)=\beta\rho(\alpha^{-1}(x))(\beta^{-1}(a),\beta^{-1}(b)),~~
T_{(\alpha,\beta)}(x)(a,b)=\beta T(\alpha^{-1}(x))(\beta^{-1}(a),\beta^{-1}(b)),\end{equation}
\begin{equation}\label{Inc4}\mu_{(\alpha,\beta)}(x)a=\beta\mu(\alpha^{-1}(x))\beta^{-1}(a),\end{equation}
for all $x,y,z\in \mathfrak g,a,b\in \mathfrak h.$

\begin{pro} With the above notations,
$(\chi_{(\alpha,\beta)},\omega_{(\alpha,\beta)},\mu_{(\alpha,\beta)},\theta_{(\alpha,\beta)},D_{(\alpha,\beta)},\rho_{(\alpha,\beta)},T_{(\alpha,\beta)})$
is a non-abelian (2,3)-cocycle. \end{pro}

\begin{proof}
By Eqs.~(\ref{L1}), (\ref{Inc1}) and (\ref{Inc3}), for all $x_1,x_2,y_1,y_2,y_3\in \mathfrak g$, we get
\begin{eqnarray*}&&D_{(\alpha,\beta)}(x_1,x_2)\omega_{(\alpha,\beta)}(y_1,y_2,y_3)+\omega_{(\alpha,\beta)}(x_1,x_2,\{y_1,y_2,y_3\}_{\mathfrak
g})\\&=&\beta
D(\alpha^{-1}(x_1),\alpha^{-1}(x_2))\beta^{-1}\beta\omega_{(\alpha,\beta)}(\alpha^{-1}(y_1),\alpha^{-1}(y_2),\alpha^{-1}(y_3))
\\&&-\beta\omega_{(\alpha,\beta)}(\alpha^{-1}(x_1),\alpha^{-1}(x_2),\alpha^{-1}(\{y_1,y_2,y_3\}_{\mathfrak
g}))
\\&=&
\beta\omega(\{\alpha^{-1}(x_1),\alpha^{-1}(x_2),\alpha^{-1}(y_1)\}_{\mathfrak
g},\alpha^{-1}(y_2),\alpha^{-1}(y_3))\\&&+\beta\theta(\alpha^{-1}(y_2),\alpha^{-1}(y_3))\omega(\alpha^{-1}(x_1),\alpha^{-1}(x_2),\alpha^{-1}(y_1))
\\&&+\beta\omega(\alpha^{-1}(y_1),\{\alpha^{-1}(x_1),\alpha^{-1}(x_2),\alpha^{-1}(y_2)\}_{\mathfrak
g},\alpha^{-1}(y_3))\\&&-\beta\theta(\alpha^{-1}(y_1),\alpha^{-1}(y_3))\omega(\alpha^{-1}(x_1),\alpha^{-1}(x_2),\alpha^{-1}(y_2))
\\&&+\beta\omega(\alpha^{-1}(y_1),\alpha^{-1}(y_2),\{\alpha^{-1}(x_1),\alpha^{-1}(x_2),\alpha^{-1}(y_3)\}_{\mathfrak
g})\\&&+\beta
D(\alpha^{-1}(y_1),\alpha^{-1}(y_2))\omega(\alpha^{-1}(x_1),\alpha^{-1}(x_2),\alpha^{-1}(y_3))
\\&=&\omega_{(\alpha,\beta)}(\{x_1,x_2,y_1\}_{\mathfrak
g},y_2,y_3)+\theta_{(\alpha,\beta)}(y_2,y_3)\omega_{(\alpha,\beta)}(x_1,x_2,y_1)+\omega_{(\alpha,\beta)}(y_1,\{x_1,x_2,y_2\}_{\mathfrak
g},y_3)\\&&-\theta_{(\alpha,\beta)}(y_1,y_3)\omega_{(\alpha,\beta)}(x_1,x_2,y_2)+\omega_{(\alpha,\beta)}(y_1,y_2,\{x_1,x_2,y_3\}_{\mathfrak
g})+D_{(\alpha,\beta)}(y_1,y_2)\omega_{(\alpha,\beta)}(x_1,x_2,y_3),
 \end{eqnarray*}
which implies that Eq.~(\ref{L1}) holds for
$(\omega_{(\alpha,\beta)},\theta_{(\alpha,\beta)})$. Similarly, we can check that
Eqs.~(\ref{L12})-(\ref{L35}) hold. This completes the proof.
\end{proof}

\begin{thm} \label{Eth1} Let $0\longrightarrow\mathfrak h\stackrel{i}{\longrightarrow}
\hat{\mathfrak g}\stackrel{p}{\longrightarrow}\mathfrak
g\longrightarrow0$ be a non-abelian extension of a Lie-Yamaguti algebra
$\mathfrak g$
 by $\mathfrak h$ with a section $s$ of $p$ and $(\chi,\omega,\mu,\theta,D,\rho,T)$ be the
corresponding non-abelian (2,3)-cocycle induced by $s$. A pair $(\alpha,\beta)\in \mathrm{Aut}(\mathfrak
g)\times \mathrm{Aut}(\mathfrak h)$ is extensible if and only if the
non-abelian (2,3)-cocycles $(\chi,\omega,\mu,\theta,D,\rho,T)$ and
$(\chi_{(\alpha,\beta)},\omega_{(\alpha,\beta)},\mu_{(\alpha,\beta)},\theta_{(\alpha,\beta)},
D_{(\alpha,\beta)},\rho_{(\alpha,\beta)},T_{(\alpha,\beta)})$
are equivalent.
\end{thm}

\begin{proof}
Suppose $(\alpha,\beta)\in \mathrm{Aut}(\mathfrak g)\times \mathrm{Aut}(\mathfrak h)$
is extensible, by Theorem ~\ref{EC}, there is a linear map
$\varphi:\mathfrak g\longrightarrow \mathfrak h$ satisfying
Eqs.~~(\ref{Iam1})-(\ref{Iam7}). For all $x,y\in \mathfrak g,a\in
\mathfrak h$, there exist $x_0,y_0\in \mathfrak g,a_0\in \mathfrak
h$ such that $x=\alpha(x_0),y=\alpha(y_0),a=\beta(a_0)$. Thus,
by Eqs.~(\ref{Iam3}), (\ref{Inc3}) and (\ref{Inc5}), we have
\begin{eqnarray*}&&
\theta_{(\alpha,\beta)}(x,y)a-\theta(x,y)a\\&=&\beta(\theta(\alpha^{-1}(x),\alpha^{-1}(y))\beta^{-1}(a))
-\theta(x,y)a
\\&=& \beta(\theta(x_0,y_0)a_0)-\theta(\alpha(x_0),\alpha(y_0))\beta(a_0)
\\&=&\{\beta(a_0),\varphi(x_0),\varphi(y_0)\}_{\mathfrak
h}-T(\alpha(y_0))(\beta(a_0),\varphi(x_0))+\rho(\alpha(x_0))(\beta(a_0),\varphi(y_0))
\\&=&\{a,\varphi\alpha^{-1}(x),\varphi\alpha^{-1}(y)\}_{\mathfrak
h}-T(y)(a,\varphi\alpha^{-1}(x))+\rho(x)(a,\varphi\alpha^{-1}(y)),
 \end{eqnarray*}
which indicates that Eq.~(\ref{E4}) holds. Analogously,
Eqs.~(\ref{E1})-(\ref{E3}) and (\ref{E5})-(\ref{E6}) hold. Thus,
$(\chi,\omega,\mu,\theta,D,\rho,T)$ and
$(\chi_{(\alpha,\beta)},\omega_{(\alpha,\beta)},\mu_{(\alpha,\beta)},\theta_{(\alpha,\beta)},
D_{(\alpha,\beta)},\rho_{(\alpha,\beta)},T_{(\alpha,\beta)})$
are equivalent via the linear map
$\varphi\alpha^{-1}:\mathfrak g\longrightarrow \mathfrak h$.

The converse part can be obtained analogously.

\end{proof}

%%%%%%%%%%%%%%%%%%%%%%%%%%%%%%%%%%%%%%%%%%%%%%%%%%%%%%%%%%%%%%%%%%%%%%%%%%%%%%%%%%%%%%%%%%%%%%%%%%%%%%%%%%
\section{Wells exact sequences for Lie-Yamaguti algebras}
In this section, we consider the Wells map associated with non-abelian extensions of Lie-Yamaguti algebras.
Then we interpret the results gained in Section 5 in terms of the Wells map.

Let $\mathcal{E}:0\longrightarrow\mathfrak h\stackrel{i}{\longrightarrow}
\hat{\mathfrak g}\stackrel{p}{\longrightarrow}\mathfrak
g\longrightarrow0$
be a non-abelian extension of $\mathfrak
g$ by $\mathfrak
h$ with a section $s$ of $p$.
Then there is a linear map $t:\hat{\mathfrak g}\longrightarrow \mathfrak
h$, such that
\begin{equation}\label{W0}it+sp=I_{\hat{\mathfrak g}}.\end{equation}
Assume that $(\chi,\omega,\mu,\theta,D,\rho,T)$ is the corresponding non-abelian (2,3)-cocycle
induced by $s$.
Define a linear map $W:\mathrm{Aut}(\mathfrak
g)\times \mathrm{Aut}(\mathfrak
h)\longrightarrow H^{(2,3)}_{nab}(\mathfrak
g,\mathfrak
h)$ by
\begin{equation}\label{W1}
	W(\alpha,\beta)=[(\chi_{(\alpha,\beta)},\omega_{(\alpha,\beta)},\mu_{(\alpha,\beta)},\theta_{(\alpha,\beta)},D_{(\alpha,\beta)},
\rho_{(\alpha,\beta)},T_{(\alpha,\beta)})
-(\chi,\omega,\mu,\theta,D,\rho,T)].
\end{equation}
The map $W$ is called the Wells map.

 \begin{pro} \label{Wm1}
	The Wells map $W$ does not depend on the choice of sections. \end{pro}

\begin{proof} For all $x,y,z\in \mathfrak g$, there are elements $x_0,y_0,z_0\in \mathfrak g$ such that $x=\alpha(x_0),y=\alpha(y_0),z=\alpha(z_0)$.
Assume that $(\chi',\omega',\mu',\theta',D^{'},\rho^{'},T^{'}) $
is another non-abelian (2,3)-cocycle corresponding to the non-abelian extension $\mathcal{E}$.
 By Lemma \ref{Le1}, we know that $(\chi',\omega',\mu',\theta',D^{'},\rho^{'},T^{'})$
 and $(\chi,\omega,\mu,\theta,D,\rho,T)$ are
 equivalent non-abelian (2,3)-cocycles via a linear map
 $\varphi:\mathfrak g\longrightarrow \mathfrak h$
According to Eqs.~(\ref{E2}) and (\ref{Inc1})-(\ref{Inc5}),
denote $\psi=\beta\varphi\alpha^{-1}$, which follows that
	\begin{align*}
		&\omega^{'}_{(\alpha,\beta)}(x,y,z)-\omega_{(\alpha,\beta)}(x,y,z)
\\=&\beta\omega^{'}(\alpha^{-1}(x),\alpha^{-1}(y),\alpha^{-1}(z))-\beta\omega(\alpha^{-1}(x),\alpha^{-1}(y),\alpha^{-1}(z))\\
		=&\beta\omega^{'}(x_0,y_0,z_0)-\beta\omega(x_0,y_0,z_0)\\
		=& \beta\Big(\theta(x_0,z_0)\varphi(y_0)-D(x_0,y_0)\varphi(z_0)
+\rho(x_0)(\varphi(y_0),\varphi(z_0))-\theta(y_0,z_0)\varphi(x_0)
\\&+T(z_0)(\varphi(x_0),\varphi(y_0))-\rho(y_0)(\varphi(x_0),\varphi(z_0))-\{\varphi(x_0),\varphi(y_0),\varphi(z_0)\}_{\mathfrak
h}+\varphi\{x_0,y_0,z_0\}_{\mathfrak
g}\Big)
\\=& \beta\Big(\theta(\alpha^{-1}(x), \alpha^{-1}(z))\beta^{-1}\psi(y)-D(\alpha^{-1}(x), \alpha^{-1}(y))\beta^{-1}\psi(z)
+\rho(\alpha^{-1}(x))(\beta^{-1}\psi(y),\beta^{-1}\psi(z))\\&-\theta(\alpha^{-1}(y), \alpha^{-1}(z))\beta^{-1}\psi(x)
+T(\alpha^{-1}(z))(\beta^{-1}\psi(x),\beta^{-1}\psi(y))-\rho(\alpha^{-1}(y))(\beta^{-1}\psi(x),\beta^{-1}\psi(z))\Big)
\\&-\{\psi(x),\psi(y),\psi(z)\}_{\mathfrak
h}+\psi(\{x, y,z\}_{\mathfrak
g})
\\=& \theta_{(\alpha,\beta)}(x,z)\psi(y)-D_{(\alpha,\beta)}(x,y)\psi(z)
+\rho_{(\alpha,\beta)}(x)(\psi(y),\psi(z))-\theta_{(\alpha,\beta)}(y, z)\psi(x)
\\&+T_{(\alpha,\beta)}(z)(\psi(x),\psi(y))-\rho_{(\alpha,\beta)}(y)(\psi(x),\psi(z))
-\{\psi(x),\psi(y),\psi(z)\}_{\mathfrak
h}+\psi(\{x, y,z\}_{\mathfrak
g}).
\end{align*}
By the same token,
\begin{align*}&\chi^{'}_{(\alpha,\beta)}(x,y)-\chi_{(\alpha,\beta)}(x,y)=[\psi(x),\psi(y)]_{\mathfrak h}+\psi([x,y]_{\mathfrak g})
 -\mu_{(\alpha,\beta)}(x)\psi(y)+\mu_{(\alpha,\beta)}(y)\psi(x),
\\&\theta^{'}_{(\alpha,\beta)}(x,y)a-\theta_{(\alpha,\beta)}(x,y)a
   =\rho_{(\alpha,\beta)}(x)(a,\psi(y))-T(y)(a,\psi(x))+\{a,\psi(x),\psi(y)\}_{\mathfrak
h},
\\&D^{'}_{(\alpha,\beta)}(x,y)a-D_{(\alpha,\beta)}(x,y)a
 =\rho_{(\alpha,\beta)}(y)(\psi(x),a)-\rho_{(\alpha,\beta)}(x)(\psi(y),a)
 +\{\psi(x),\psi(y),a\}_{\mathfrak
h},
 \\&\mu^{'}_{(\alpha,\beta)}(x)a-\mu_{(\alpha,\beta)}(x)a=[a,\psi(x)]_{\mathfrak
h},~~~~\rho^{'}_{(\alpha,\beta)}(x)(a,b)-\rho_{(\alpha,\beta)}(x)(a,b)=\{a,\psi(x),b\}_{\mathfrak
h}, \\&T^{'}_{(\alpha,\beta)}(x)(a,b)-T_{(\alpha,\beta)}(x)(a,b)=\{b,a,\psi(x)\}_{\mathfrak
h}.\end{align*}
So, $(\chi'_{(\alpha,\beta)},\omega'_{(\alpha,\beta)},\mu'_{(\alpha,\beta)},\theta'_{(\alpha,\beta)},
D'_{(\alpha,\beta)},\rho'_{(\alpha,\beta)},T'_{(\alpha,\beta)})$
and $(\chi_{(\alpha,\beta)},\omega_{(\alpha,\beta)},\mu_{(\alpha,\beta)},\theta_{(\alpha,\beta)},
D_{(\alpha,\beta)},\rho_{(\alpha,\beta)},T_{(\alpha,\beta)})$
 are equivalent non-abelian (2,3)-cocycles via the linear map $\psi=\beta\varphi\alpha^{-1}$.
Combining Lemma \ref{Le1}, we know that
 \begin{align*}(\chi'_{(\alpha,\beta)},\omega'_{(\alpha,\beta)},\mu'_{(\alpha,\beta)},\theta'_{(\alpha,\beta)},
D'_{(\alpha,\beta)},\rho'_{(\alpha,\beta)},T'_{(\alpha,\beta)})-(\chi',\omega',\mu',\theta',D^{'},\rho^{'},T^{'})\end{align*}
 and \begin{align*} (\chi_{(\alpha,\beta)},\omega_{(\alpha,\beta)},\mu_{(\alpha,\beta)},\theta_{(\alpha,\beta)},
D_{(\alpha,\beta)},\rho_{(\alpha,\beta)},T_{(\alpha,\beta)})-(\chi,\omega,\mu,\theta,D,\rho,T)\end{align*}
 are equivalent via the linear map $\beta\varphi\alpha^{-1}-\varphi$.
\end{proof}

  \begin{pro} \label{Wm2} Let
$\mathcal{E}:0\longrightarrow\mathfrak h\stackrel{i}{\longrightarrow}
\hat{\mathfrak g}\stackrel{p}{\longrightarrow}\mathfrak g\longrightarrow0$
be a non-abelian extension of $\mathfrak g$ by $\mathfrak h$  with a section $s$ of $p$. Define a map
\begin{equation}\label{W3}K:\mathrm{Aut}_{\mathfrak h}(\hat{\mathfrak g})\longrightarrow \mathrm{Aut}(\mathfrak g)\times \mathrm{Aut}(\mathfrak g),~~K(\gamma)=(p\gamma s,\gamma|_{\mathfrak h}),~\forall~\gamma\in \mathrm{Aut}_{\mathfrak h}(\hat{\mathfrak g}).\end{equation}
 Then $K$ is a homomorphism of groups. \end{pro}

\begin{proof}
One can take the same procedure of Lie algebras, see \cite{1}.
\end{proof}

 \begin{thm} \label{Wm3} Assume that
$\mathcal{E}:0\longrightarrow\mathfrak h\stackrel{i}{\longrightarrow}
\hat{\mathfrak g}\stackrel{p}{\longrightarrow}\mathfrak g\longrightarrow0$
is a non-abelian extension of $\mathfrak g$ by $\mathfrak h$ with a section $s$ of $\hat{\mathfrak g}$.
Then there is an exact sequence:
$$1\longrightarrow \mathrm{Aut}_{\mathfrak h}^{\mathfrak g}(\hat{\mathfrak g})\stackrel{H}{\longrightarrow} \mathrm{Aut}_{\mathfrak h}(\hat{\mathfrak g})\stackrel{K}{\longrightarrow}\mathrm{Aut}(\mathfrak g)\times \mathrm{Aut}(\mathfrak h)\stackrel{W}{\longrightarrow} H^{(2,3)}_{nab}(\mathfrak g,\mathfrak h),$$
where $\mathrm{Aut}_{\mathfrak h}^{\mathfrak g}(\hat{\mathfrak g})=\{\gamma \in \mathrm{Aut}(\hat{\mathfrak g})| K(\gamma)=(I_{\mathfrak g},I_{\mathfrak h}) \}$.\end{thm}

\begin{proof} Obviously, $\mathrm{Ker} K=\mathrm{Im}H$ and $H$ is injective. We only need to prove that $\mathrm{Ker} W=\mathrm{Im}K$.
By Theorem \ref{Eth1}, for all $(\alpha,\beta)\in \mathrm{Ker} W$, we get that $(\alpha,\beta)$ is extensible
 with respect to the non-abelian extension $\mathcal{E}$, that is, there is a $\gamma\in \mathrm{Aut}_{\mathfrak h}^{\mathfrak g}(\hat{\mathfrak g})$, such that
$i\beta=\gamma i,~p\gamma=\alpha p$, which follows that $\alpha=\alpha p s=p\gamma s,~\beta=\gamma|_{\mathfrak h}.$
Thus,  $(\alpha,\beta)\in \mathrm{Im}K$. On the other hand, for any $(\alpha,\beta)\in \mathrm{Im}K$, there is an isomorphism
$\gamma\in \mathrm{Aut}_{\mathfrak h}(\hat{\mathfrak g})$, such that
 Eq.~(\ref{W3}) holds. Combining Eq.~(\ref{W0}) and $\mathrm{Im}i=\mathrm{Ker }p$, we obtain
 $\alpha p=p\gamma s p=p\gamma(I_{\hat{\mathfrak g}}-it)=p\gamma $ and $i\beta=\gamma i$. Hence, $(\alpha,\beta)$ is extensible
 with respect to the non-abelian extension $\mathcal{E}$. According to Theorem \ref{Eth1},
 $(\alpha,\beta)\in \mathrm{Ker} W$. In all, $\mathrm{Ker} W=\mathrm{Im}K$.
\end{proof}
Let
$\mathcal{E}:0\longrightarrow\mathfrak h\stackrel{i}{\longrightarrow}
\hat{\mathfrak g}\stackrel{p}{\longrightarrow}\mathfrak g\longrightarrow0$
be a non-abelian extension of $\mathfrak g$ by $\mathfrak h$  with a section $s$ of $p$. Suppose that
$(\chi,\omega,\mu,\theta,D,\rho,T)$ is a non-abelian (2,3)-cocycle induced by the section $s$.

Denote
\begin{align}
		Z_{nab}^{1}(\mathfrak g,\mathfrak h)=&\left\{\varphi:\mathfrak g\rightarrow \mathfrak h\left|\begin{aligned}&[a,\varphi(x)]_{\mathfrak
     h}=\{a,b,\varphi(x)\}_{\mathfrak
     h}=\{\varphi(x),a,b\}_{\mathfrak h}=0,\\&\{\varphi(x),\varphi(y),a\}_{\mathfrak
     h}-\rho(x)(\varphi(y),a)+\rho(y)(\varphi(x),a)=0,
     \\&\{a,\varphi(x),\varphi(y)\}_{\mathfrak h}-T(y)(a,\varphi(x))+\rho(x)(a,\varphi(y))=0,\\&
    \mu(x)\varphi(y)-\mu(y)\varphi(x)= \varphi([x,y]_{\mathfrak g})+[\varphi(x),\varphi(y)]_{\mathfrak h},~
    \\&T(z)(\varphi(x),\varphi(y))-\rho(y)(\varphi(x),\varphi(z))
     -\theta(y,z)\varphi(x)\\&+\rho(x)(\varphi(y),\varphi(z))+\theta(x,z)\varphi(y)
-D(x,y)\varphi(z)\\&=\{\varphi(x),\varphi(y),\varphi(z)\}_{\mathfrak
     h}-\varphi(\{x,y,z\}_{\mathfrak h}),~\forall~x,y,z\in {\mathfrak g},a,b\in {\mathfrak h}
     \end{aligned}\right.\right\}.\label{W5}
	\end{align}
It is easy to check that $Z_{nab}^{1}(\mathfrak g,\mathfrak h)$ is an abelian group, which is called a non-abelian 1-cocycle on $\mathfrak g$ with values in $\mathfrak h$.
  \begin{pro} \label{Wm4} With the above notations, we have
\begin{enumerate}[label=$(\roman*)$,leftmargin=15pt]
  \item The linear map $S:\mathrm{Ker} K\longrightarrow Z_{nab}^{1}(\mathfrak g,\mathfrak h)$ defined by
 \begin{equation}\label{W6}S(\gamma)(x)=\varphi_{\gamma}(x)=s(x)-\gamma s(x),~\forall~~\gamma\in \mathrm{Ker} K,~x \in \mathfrak g\end{equation} is
a homomorphism of groups.
\item $S$ is an isomorphism, that is, $\mathrm{Ker K}\simeq Z_{nab}^{1}(\mathfrak g,\mathfrak h)$.
\end{enumerate}
\end{pro}

\begin{proof}
\begin{enumerate}[label=$(\roman*)$,leftmargin=15pt]
	\item
	By Eqs.~(\ref{C1})-(\ref{C3}), (\ref{W3}) and (\ref{W6}), for all $x,y,z\in \mathfrak g$, we have,
\begin{align*}&\{\varphi_{\gamma}(x),\varphi_{\gamma}(y),\varphi_{\gamma}(z)\}_{\mathfrak h}-T(z)(\varphi_{\gamma}(x),\varphi_{\gamma}(y))+\rho(y)(\varphi_{\gamma}(x),\varphi_{\gamma}(z))+
\theta(y,z)\varphi_{\gamma}(x)\\&-\rho(x)(\varphi_{\gamma}(y),\varphi_{\gamma}(z))-\theta(x,z)\varphi_{\gamma}(y)
+D(x,y)\varphi_{\gamma}(z)-\varphi_{\gamma}(\{x,y,z\}_{\mathfrak g})
\\=&\{s(x)-\gamma s(x),s(y)-\gamma s(y),s(z)-\gamma s(z)\}_{\hat{\mathfrak g}}-\{s(x)-\gamma s(x),s(y)-\gamma s(y),s(z)\}_{\hat{\mathfrak g}}
\\&+\{s(y),s(x)-\gamma s(x),s(z)-\gamma s(z)\}_{\hat{\mathfrak g}}+\{s(x)-\gamma s(x),s(y),s(z)\}_{\hat{\mathfrak g}}
-\{s(x),s(y)-\gamma s(y),s(z)-\gamma s(z)\}_{\hat{\mathfrak g}}\\&-\{s(y)-\gamma s(y),s(x),s(z)\}_{\hat{\mathfrak g}}
+\{s(x),s(y),s(z)-\gamma s(z)\}_{\hat{\mathfrak g}}+\gamma s(\{x,y,z\}_{\mathfrak g})-s(\{x,y,z\}_{\mathfrak g})
\\=&\gamma s(\{x,y,z\}_{\mathfrak g})-\{\gamma s(x),\gamma s(y),\gamma s(z)\}_{\hat{\mathfrak g}}+\{s(x),s(y),s(z)\}_{\hat{\mathfrak g}}-s(\{x,y,z\}_{\mathfrak g})
\\=&\omega(x, y,z)-\gamma \omega(x, y,z)
\\=&0.
\end{align*}
Analogously, we can check that $\varphi_{\gamma}$ satisfies the other identities in $ Z_{nab}^{1}(\mathfrak g,\mathfrak h)$.
Thus, $S$ is well-defined.
For any $\gamma_1,\gamma_2\in \mathrm{Ker} K$ and $x\in \mathfrak g$, suppose $S(\gamma_1)=\varphi_{\gamma_1}$ and $S(\gamma_2)=\varphi_{\gamma_2}$.
By Eqs.~ (\ref{W3}) and (\ref{W6}), we have
\begin{align*}S(\gamma_1 \gamma_2)(x)&=s(x)-\gamma_1 \gamma_2s(x)
\\&=s(x)-\gamma_1(s(x)-\varphi_{\gamma_2}(x))
\\&=s(x)-\gamma_1s(x)+\gamma_{1}\varphi_{\gamma_2}(x)
\\&=\varphi_{\gamma_1}(x)+\varphi_{\gamma_2}(x)\end{align*}
which means that $S(\gamma_1 \gamma_2)=S(\gamma_1)+S( \gamma_2)$ is a homomorphism of groups.

\item For all $\gamma\in \mathrm{Ker}K$, we obtain that $K(\gamma)=(p\gamma s,\gamma|_{\mathfrak h})=(I_\mathfrak g,I_\mathfrak h)$.
 If $S(\gamma)=\varphi_{\gamma}=0$,
we can get $\varphi_{\gamma}(x)=s(x)-\gamma s(x)=0$, that is, $\gamma=I_{\hat{\mathfrak g}}$, which indicates that $S$ is injective.
Secondly, we prove that $S$ is surjective. Since $s$ is a section of $p$, all $\hat{x}\in \hat{\mathfrak g}$ can be written as $a+s(x)$ for
 some $a\in \mathfrak h, x\in \mathfrak g$.
 For any $\varphi\in Z_{nab}^{1}(\mathfrak g,\mathfrak h)$, define a linear map $\gamma:\hat{\mathfrak g}\rightarrow \hat{\mathfrak g}$ by
  \begin{equation}\label{W7}\gamma(\hat{x})=\gamma(a+s(x))=s(x)-\varphi(x)+a,~\forall~\hat{x}\in \hat{\mathfrak g}.\end{equation}
It is obviously that $(p\gamma s,\gamma|_{\mathfrak h})=(I_\mathfrak g,I_\mathfrak h)$.
We need to verify that $\gamma $ is an automorphism of Lie-Yamaguti algebra $\hat{\mathfrak g}$. One can take the same
procedure as the proof of the converse part of Theorem \ref{EC}.
 It follows that $\gamma\in \mathrm{Ker} K$.
Thus, $S$ is surjective. In all, $S$ is bijective.
 So $\mathrm{Ker }K\simeq Z_{nab}^{1}(\mathfrak g,\mathfrak h)$.
 \end{enumerate}
\end{proof}

Combining Theorem \ref{Wm3} and Proposition \ref{Wm4}, we have

\begin{thm} \label{Wm5} Let
$\mathcal{E}:0\longrightarrow\mathfrak h\stackrel{i}{\longrightarrow}
\hat{\mathfrak g}\stackrel{p}{\longrightarrow}\mathfrak g\longrightarrow0$
be a non-abelian extension of $\mathfrak g$ by $\mathfrak h$. There is an exact sequence:
$$0\longrightarrow Z_{nab}^{1}(\mathfrak g,\mathfrak h)\stackrel{i}{\longrightarrow} \mathrm{Aut}_{\mathfrak h}(\hat{\mathfrak g})\stackrel{K}{\longrightarrow}\mathrm{Aut}(\mathfrak g)\times \mathrm{Aut}(\mathfrak h)\stackrel{W}{\longrightarrow} H^{(2,3)}_{nab}(\mathfrak g,\mathfrak h).$$
\end{thm}

%%%%%%%%%%%%%%%%%%%%%%%%%%%%%%%%%%%%%%%%%%%%%%%%%%%%%%%%%%%%%%%%%%%%%%%%%%%%%%%%%%%%%%%%%%%%%%%%%%%%%%%
\section{Particular case: abelian extensions of Lie-Yamaguti algebras}
 In this section, we discuss the results of previous section in particular case.

We fix the abelian extension
$\mathcal{E}:0\longrightarrow\mathfrak h\stackrel{i}{\longrightarrow} \hat{\mathfrak g}\stackrel{p}{\longrightarrow}\mathfrak g\longrightarrow0$
of $\mathfrak g$ by $\mathfrak h$ with a section $s$ of $p$.
Assume that $(\chi,\omega)$ is a (2,3)-cocycle corresponding to $\mathcal{E}$.
In the case of abelian extensions $\mathcal{E}$, the maps $\rho, T$ defined by (\ref{C3}) become to zero. Then
the quadruple $(\mathfrak h,\mu,\theta,D)$ given by Eq.~(\ref{C2}) is a representation of $\mathfrak g$ \cite{31}. Moreover,

\begin{thm}[\cite{31}]
	\begin{enumerate}[label=$(\roman*)$,leftmargin=15pt]
	\item The triple $(\mathfrak g\oplus \mathfrak h,[ \ , \  ]_{\chi},\{ \ , \ , \  \}_{\omega})$ is a Lie-Yamaguti algebra
 if and only if $(\chi,\omega)$ is a (2,3)-cocycle of $\mathfrak g$ with coefficients in the
representation $(\mathfrak h,\mu,\theta,D)$.

\item Abelian extensions of $\mathfrak g$
 by $\mathfrak h$ are classified
by the cohomology group $H^{(2,3)}(\mathfrak g,\mathfrak h)$ of
$\mathfrak g$ with coefficients in $(\mathfrak h,\mu,\theta,D)$.
\end{enumerate}
\end{thm}

\begin{thm} Let $\mathcal{E}:0\longrightarrow\mathfrak h\stackrel{i}{\longrightarrow} \hat{\mathfrak g}
\stackrel{p}{\longrightarrow}\mathfrak g\longrightarrow0$ be an abelian extension of $\mathfrak g$
by $\mathfrak h$ with a section $s$ of $p$. Assume that $(\chi,\omega)$ is a (2,3)-cocycle
and $(\mathfrak h,\mu,\theta,D)$ is a representation of $\mathfrak g$
associated to $\mathcal{E}$. A pair $(\alpha,\beta)\in \mathrm{Aut}(\mathfrak g
)\times \mathrm{Aut}(\mathfrak h)$ is extensible with respect to the abelian extension $\mathcal{E}$ if and only if there is a
linear map $\varphi:\mathfrak g\longrightarrow \mathfrak h$
satisfying the following conditions:
\begin{align}
     \beta\omega(x,y,z)-\omega(\alpha(x),\alpha(y),\alpha(z))=&\nonumber
\theta(\alpha(x),\alpha(z))\varphi(y)-\theta(\alpha(y),\alpha(z))\varphi(x)
\\&\label{AEE1}-D(\alpha(x),\alpha(y))\varphi(z)+\varphi(\{x,y,z\}_{\mathfrak
     h}),
\end{align}
\begin{equation}\label{AEE2}
     \beta \chi(x,y)-\chi(\alpha(x),\alpha(y))=
    \mu(\alpha(y))\varphi(x)- \mu(\alpha(x))\varphi(y)+\varphi([x,y]_{\mathfrak g}),
\end{equation}
\begin{equation}\label{AEE3}
     \beta(\theta(x,y)a)=\theta(\alpha(x),\alpha(y))\beta(a),~~\beta \mu(x)a=\mu(\alpha(x))\beta(a).
\end{equation}

\end{thm}

\begin{proof}
It can be obtained directly from Theorem \ref{EC}.
\end{proof}
By Eqs.~(\ref{eq2.13}) and (\ref{AEE3}), we obtain
\begin{equation*}\beta D(x,y)a=D(\alpha(x),\alpha(y))\beta(a).
\end{equation*}

For all $(\alpha,\beta)\in \mathrm{Aut}(\mathfrak g
)\times \mathrm{Aut}(\mathfrak h)$, $(\chi_{(\alpha,\beta)},\omega_{(\alpha,\beta)})$ may not be a (2,3)-cocycle.
 Indeed,  $(\chi_{(\alpha,\beta)},\omega_{(\alpha,\beta)})$ is a (2,3)-cocycle if Eq.~(\ref{AEE3}) holds.
 Thus, it is natural to introduce the space of compatible pairs of automorphisms:
\begin{align*}
		C_{(\mu,\theta)}=&\left\{(\alpha,\beta)\in \mathrm{Aut}(\mathfrak g
)\times \mathrm{Aut}(\mathfrak h)\left|\begin{aligned}&\beta(\theta(x,y)a)=\theta(\alpha(x),\alpha(y))\beta(a),
\\&\beta \mu(x)a=\mu(\alpha(x))\beta(a),~\forall~x,y\in {\mathfrak g},a\in {\mathfrak h}
     \end{aligned}\right.\right\}.
	\end{align*}
More detail on the space of compatible pairs of automorphisms can be found in \cite{021}.

Analogous to Theorem \ref{Eth1}, we get

\begin{thm}[\cite{021}]\label{Wm6} Let $\mathcal{E}:0\longrightarrow\mathfrak h\stackrel{i}{\longrightarrow} \hat{\mathfrak g}
\stackrel{p}{\longrightarrow}\mathfrak g\longrightarrow0$ be an abelian extension of $\mathfrak g$
by $\mathfrak h$ with a section $s$ of $p$ and $(\chi,\omega)$ be a
 (2,3)-cocycle associated to $\mathcal{E}$. A pair $(\alpha,\beta)\in C_{(\mu,\theta)}$ is extensible with
  respect to the abelian extension
$\mathcal{E}$ if and only if  $(\chi,\omega)$ and $(\chi_{(\alpha,\beta)},\omega_{(\alpha,\beta)})$ are in the same cohomological class.
\end{thm}

In the case of abelian extensions, $Z^{1}_{nab}(\mathfrak g,\mathfrak h)$
defined by (\ref{W5}) becomes to ${H}^{1}(\mathfrak g,\mathfrak h)$ given in Section 2.
 In the light of Theorem \ref{Wm5} and Theorem \ref{Wm6}, we have the following exact sequence:

\begin{thm}[\cite{021}] Let $\mathcal{E}:0\longrightarrow\mathfrak h\stackrel{i}{\longrightarrow} \hat{\mathfrak g}
\stackrel{p}{\longrightarrow}\mathfrak g\longrightarrow0$ be an abelian extension of $\mathfrak g$
by $\mathfrak h$. There is an exact sequence:
$$0\longrightarrow H^{1}(\mathfrak g,\mathfrak h)\stackrel{i}{\longrightarrow} \mathrm{Aut}_{\mathfrak h}(\hat{\mathfrak g})\stackrel{K}{\longrightarrow}C_{(\mu,\theta)}\stackrel{W}{\longrightarrow} H^{(2,3)}(\mathfrak g,\mathfrak h).$$
\end{thm}

%%%%%%%%%%%%%%%%%%%%%%%%%%%%%%%%%%%%%%%%%%%%%%%%%%%%%%%%%%%%%%%%%%%%%%%%%%%%%%%%%%%%%%%%%%%%%%%%%%%%%%%%%%

\begin{center}{\textbf{Acknowledgments}}
\end{center}
This work was supported by the National
Natural Science Foundation of China (11871421); Natural Science
Foundation of Zhejiang Province of China (LY19A010001); and Science
and Technology Planning Project of Zhejiang Province
(2022C01118).

%%%%%%%%%%%%%%%%%%%%%%%%%%%%%%%%%%%%%%%%%%%%%%%%%%%%%%%%%%%%%%%%%%%%%%%%%%%%%%%%%%%%%%%%%%%%%%%%%%%%%%%%%%
\begin{center} {\textbf{Statements and Declarations}}
\end{center}
 All datasets underlying the conclusions of the paper are available
to readers. No conflict of interest exits in the submission of this
manuscript.

%%%%%%%%%%%%%%%%%%%%%%%%%%%%%%%%%%%%%%%%%%%%%%%%%%%%%%%%%%%%%%%%%%%%%%%%%%%%%%%%%%%%%%%%%%%%%%%%%%%%%%%%%%

\end{document}